\documentclass[11pt,a4paper,reqno]{amsart}

\usepackage{cmap}
\usepackage[T1]{fontenc}
\usepackage[utf8]{inputenc}

\usepackage{diagbox}
\usepackage{makecell}

\usepackage[normalem]{ulem}
\usepackage{graphicx}
\usepackage{epstopdf}
\usepackage{subcaption}
\usepackage{marginnote} 

\usepackage{setspace,mathrsfs,amsthm,amsmath,amssymb,amsfonts,lipsum,appendix,mathtools,cite,tabularx,fontenc, enumerate}

\usepackage[shortlabels]{enumitem}

\makeatletter
\newtheorem*{rep@theorem}{\rep@title}
\newcommand{\newreptheorem}[2]{%
	\newenvironment{rep#1}[1]{%
		\def\rep@title{#2 \ref{##1}}%
		\begin{rep@theorem}}%
		{\end{rep@theorem}}}
\makeatother

\newtheorem{theorem}{Theorem}[section]
\newtheorem{lemma}[theorem]{Lemma}

\newtheorem{corollary}[theorem]{Corollary}

\theoremstyle{definition}
\newtheorem{remark}[theorem]{Remark}
\newtheorem{definition}[theorem]{Definition}

\usepackage[nointegrals]{wasysym}
\usepackage[misc]{ifsym}
\usepackage[dvipsnames]{xcolor}
\definecolor{ao}{rgb}{0.0, 0.5, 0.0}
\definecolor{lasallegreen}{rgb}{0.03, 0.47, 0.19}
\usepackage{hyperref}
\hypersetup{
	colorlinks=true,
	linkcolor=blue,
	filecolor=magenta,      
	urlcolor=black,
	citecolor=OrangeRed,
}
\usepackage[margin=1in]{geometry}

\definecolor{ForestGreen}{rgb}{0.1,0.6,0.05}
\definecolor{EgyptBlue}{rgb}{0.063,0.1,0.6}

\allowdisplaybreaks

\usepackage{easyReview}

\let\oldnorm\norm
\def\norm{\@ifstar{\oldnorm}{\oldnorm*}}

\newcommand{\al} {\alpha}
\newcommand{\pa} {\partial}

\newcommand{\de} {\delta}

\newcommand{\Om} {\Omega}
\newcommand{\la} {\lambda}

\newcommand{\si} {\sigma}

\newcommand{\dx}{{\,\rm d}x}
\newcommand{\dt}{{\,\rm d}t}

\newcommand{\ds}{{\,\rm d}s}
\newcommand{\dS}{{\,\rm dS}}
\newcommand{\da}{{\,\rm d}\al}
\newcommand{\dd}{{\,\rm d}\de}
\newcommand{\dsi}{{\,\rm d}\si}
\newcommand{\drh}{{\,\rm d}\rho}

\newcommand{\R}{{\mathbb R}}

\newcommand{\Omo}{\Om_{\rm{out}}}
\newcommand{\Omi}{\Om_{\rm{in}}}

\newcommand{\Go}{G_{\rm{out}}}
\newcommand{\Gi}{G_{\rm{in}}}

\newcommand{\ho}{h_{\rm{out}}}
\newcommand{\hi}{h_{\rm{in}}}

\def\dx{{\,\rm d}x}
\def\dt{{\,\rm d}t}

\def\A{{\mathscr A}}

\usepackage[hyperpageref]{backref}

\usepackage{orcidlink}

\usepackage{csquotes}

\usepackage{xpatch}
\makeatletter   
\xpatchcmd{\@tocline}
{\hfil\hbox to\@pnumwidth{\@tocpagenum{#7}}\par}
{\ifnum#1<0\hfill\else\dotfill\fi\hbox to\@pnumwidth{\@tocpagenum{#7}}\par}
{}{}

\def\l@subsection{\@tocline{2}{0pt}{2pc}{6pc}{}}
\def\l@subsubsection{\@tocline{3}{0pt}{8pc}{8pc}{}}
\makeatother

\usepackage{amsfonts}
\usepackage{amssymb,amsthm}
\usepackage{amsmath}

\allowdisplaybreaks
\usepackage{mathtools}
\mathtoolsset{showonlyrefs}

\numberwithin{equation}{section}

\usepackage{setspace}
\usepackage{enumitem}
\setlist{nosep}

\begin{document}
\singlespacing
\title
{Reverse Faber-Krahn inequality for planar doubly connected domains}

\author[T.V.~Anoop]{T.V.~Anoop}
\author[V.~Bobkov]{Vladimir Bobkov}
\author[M.~Ghosh]{Mrityunjoy Ghosh}

\address[T.V.~Anoop]{\newline\indent
	Department of Mathematics,
	Indian Institute of Technology Madras, 
	\newline\indent
	Chennai 36, India
	\newline\indent
	\orcidlink{0000-0002-2470-9140} 0000-0002-2470-9140 
}
\email{anoop@iitm.ac.in}

\address[V.~Bobkov]{\newline\indent
	Institute of Mathematics, Ufa Federal Research Centre, RAS,
	\newline\indent 
	Chernyshevsky str. 112, 450008 Ufa, Russia
	\newline\indent 
	Ufa University of Science and Technology,
 \newline\indent
 Zaki Validi str. 32, 450076 Ufa, Russia
  \newline\indent
 \orcidlink{0000-0002-4425-0218} 0000-0002-4425-0218
}
\email{bobkov@matem.anrb.ru}

\address[M.~Ghosh]{\newline\indent
	Tata Institute of Fundamental Research,
 \newline\indent
	Centre for Applicable Mathematics, 
	\newline\indent
	Sharadanagar, Bengaluru 560065, India
	\newline\indent
	\orcidlink{0000-0003-0415-2821} 0000-0003-0415-2821 
}
\email{ghoshmrityunjoy22@gmail.com}


\begin{abstract}
We prove that among all doubly connected and elastically supported planar membranes $\Om$ with prescribed values of the area $|\Omega|$ and the lengths of the inner and outer boundaries $|\pa\Omi|_1$, $|\pa\Omo|_1$ satisfying $|\pa\Omo|_1^2 - |\pa\Omi|_1^2 = 4\pi |\Omega|$, the concentric annular membrane has the maximal fundamental frequency. 
The elastic constants $\hi$, $\ho$ on $\pa\Omi$, $\pa\Omo$, respectively, are assumed to satisfy $\hi \cdot \ho \geq 0$ and can admit negative values and $+\infty$, the latter being understood as a fixation of the membrane on the corresponding part of the boundary. 
Our study extends and unifies several existing results in the literature. 
The case $\hi \cdot \ho = 0$ is proved using the method of interior parallels \`a la \textsc{Payne \& Weinberger} \cite{payne}, and it requires less restrictive assumptions on $\Omega$.
For the case $\hi \cdot \ho > 0$, we develop the construction of the so-called ``effectless cut'' of $\Omega$ described in terms of the gradient flow of the first eigenfunction. This concept was originally introduced by \textsc{Weinberger} \cite{weineffect} and used by \textsc{Hersch} \cite{hers} in the fixed boundary case, whose arguments we also revise. 
\end{abstract} 

\maketitle

\begin{quote}
	\setcounter{tocdepth}{2}
	\tableofcontents
	\addtocontents{toc}{\vspace*{0ex}}
\end{quote}

\section{Introduction}\label{sec:intro}

Let $\Om\subset \R^2$ be a bounded domain of the form 
$\Om =\Omo\setminus \overline{\Omi}$,
where $\Omi$ and $\Omo$ are simply connected open sets such that $\overline{\Omi}\subset \Omo$. 
We always assume that $\pa \Om$ is of class $C^{1,1}$ and $\Omi$ is nonempty, unless otherwise  explicitly mentioned.  
For $\hi,\ho\in \R\cup\{+\infty\}$ we consider the following eigenvalue problem for the Robin Laplacian in $\Om$:
\begin{align}\tag{$\mathcal{RR}$}\label{cutproblem2d}
\left\{\begin{aligned}
    -\Delta u &= \lambda u \quad\text{in} \quad\Omega,\\
		\frac{\pa u}{\pa \nu}+ \hi u&=0 \quad\;\;\text{on}\quad  \pa \Omi,\\
		\frac{\pa u}{\pa \nu}+ \ho u &=0 \quad\;\; \text{on}\quad  \pa \Omo,
	\end{aligned}
	 \right.
\end{align}
where $\lambda \in \mathbb{R}$ and $\nu$ denotes the unit outward normal vector to $\pa \Om$. 
In what follows, the case $\hi=0$ or $\ho = 0$ is referred to as the Neumann boundary condition $\partial u/\partial \nu=0$, and  the case $\hi = +\infty$ or $\ho = +\infty$ as the Dirichlet boundary condition $u=0$, imposed on the respective part of $\partial \Omega$. 
Throughout the text, we denote by $B_r$ the open disk of radius $r>0$ centered at the origin, and the notation $\A_{r,R}$ is reserved for the (concentric) annulus:
$$
\A_{r,R}:=B_R\setminus\overline{B_r}
\quad \text{for}~ R>r>0.
$$ 
We will also denote by $|\cdot|$ and $|\cdot|_1$ the area of an open subset of $\mathbb{R}^2$ and the length of a curve, respectively. 

It follows from the general theory of compact self-adjoint operators (see, e.g., an exposition in \cite[Section~4.2]{BFK} on a closely related problem) that the spectrum of \eqref{cutproblem2d} is discrete and accumulates at infinity. 
In the present work, we are interested in inequalities of isoperimetric type for the \textit{first} eigenvalue of \eqref{cutproblem2d} under various mutual relations between the Robin parameters $\hi$ and $\ho$. 
The first eigenvalue can be characterized as
\begin{equation}\label{eq:lambda1}
\lambda_1(\Omega) 
=
\inf_{v \in \tilde{H}^1(\Omega)\setminus\{0\}}
\frac{\int_{\Om}|\nabla v|^2\dx+\hi \int_{\pa \Omi} v^2\dsi+\ho \int_{\pa \Omo} v^2\dsi}{\int_{\Om} v^2\dx},
\end{equation}
where
 $$
 \tilde{H}^1(\Omega):=\left\{u\in H^1(\Omega): u=0 \text{ on } \partial \Omi \text{ if } \hi=+\infty,~ u=0 \text{ on } \partial \Omo \text{ if } \ho=+\infty \right\}.
 $$
In particular, if $\hi,\ho=+\infty$, then $\tilde{H}^1(\Omega) = H_0^1(\Omega)$. 
By standard arguments based on the characterization \eqref{eq:lambda1} and the maximum principle, the first eigenvalue is simple and the corresponding eigenfunction does not vanish in $\Omega$. 
Moreover, any eigenfunction is (real) analytic in $\Omega$ and belongs to $C^1(\overline{\Omega})$, see \cite[Proposition~2.1]{hassannezhad2024pleijel} and \cite[Section~4.2]{BFK} for explicit references and discussion.
We also provide a remark on the sign of $\lambda_1(\Omega)$.
\begin{remark}\label{rem:sign}
    For $(\hi,\ho) \neq (0,0)$, it is not hard to see that the following assertions hold:
	\begin{enumerate}[label={\rm(\roman*)}]
		\item If $\hi,\ho \in [0,+\infty]$, then $\lambda_1(\Om) > 0$.
		\item If $\hi,\ho \in (-\infty, 0]$, then $\lambda_1(\Om) < 0$.
	\end{enumerate}
\end{remark}

Throughout the text, for the convenience of exposition, we adopt the following convention:
we write the first eigenvalue of a doubly connected domain $\Omega$ with two superscript letters among  `$\mathcal{R}$', `$\mathcal{N}$', `$\mathcal{D}$' (corresponding to $\mathcal{R}$obin, $\mathcal{N}$eumann, $\mathcal{D}$irichlet, respectively), where the first letter refers to the boundary condition imposed on the \textit{inner} boundary $\pa \Omi$, and the second letter refers to the boundary condition imposed on the \textit{outer} boundary $\pa \Omo$. 
In particular, we use the notation
$$
\lambda_1^{\mathcal{RR}}(\Om),~ 
\lambda_1^{\mathcal{NR}}(\Om),~ 
\lambda_1^{\mathcal{RN}}(\Om),~
\lambda_1^{\mathcal{NN}}(\Om),~
\lambda_1^{\mathcal{DN}}(\Om),~
\lambda_1^{\mathcal{ND}}(\Om),~
\lambda_1^{\mathcal{DD}}(\Om).
$$
The only exception in the text is the next paragraph, where we overview the case $h:=\hi=\ho$ and allow $\Omi$ to be empty, 
and write only one superscript: $\lambda_1^{\mathcal{R}}(\Om)$, $\lambda_1^{\mathcal{N}}(\Om)$, $\lambda_1^{\mathcal{D}}(\Om)$.

The classical Faber-Krahn inequality, due to \textsc{Faber} \cite{faber1923beweis} and \textsc{Krahn} \cite{krahn1925rayleigh}, corresponds to the Dirichlet boundary condition and states that 
\begin{equation}\label{eq:FKD}
\lambda_1^{\mathcal{D}}(B_R) \leq \lambda_1^{\mathcal{D}}(\Om), 
\end{equation}
where $R>0$ is such that $|B_R| = |\Omega|$. 
This result is valid regardless of the topology of $\Omega$.
It was proved by \textsc{Bossel} \cite{bossel1986membranes} and \textsc{Daners} \cite{daners2006faber} 
that the Faber-Krahn inequality is true also in the positive Robin case $h \in (0,+\infty)$: 
\begin{equation}\label{eq:FKRp}
\lambda_1^{\mathcal{R}}(B_R) \leq \lambda_1^{\mathcal{R}}(\Om).
\end{equation}
Evidently, the Neumann case $h=0$ is out of interest since $\lambda_1^{\mathcal{N}}(\Omega) = 0$ for any $\Omega$. 
Things become more complicated in the negative Robin case $h \in (-\infty,0)$ and still remain obscure in the full generality. 
In this direction, \textsc{Freitas \& Krej\v{c}i\v{r}\'ik} \cite{freitas2015first} proved the existence of $h^*=h^*(|\Omega|) < 0$ such that 
\begin{equation}\label{eq:FKRn}
	\lambda_1^{\mathcal{R}}(\Omega) \leq \lambda_1^{\mathcal{R}}(B_R) \quad \text{for any}~ h \in [h^*,0),
\end{equation}
while they showed that this inequality is reversed when $\Omega$ is an annulus and the Robin parameter is a sufficiently large negative number. 

The bounds \eqref{eq:FKD}, \eqref{eq:FKRp}, \eqref{eq:FKRn} are not sensitive to the topology of $\Omega$. 
However, the topology of $\Omega$ comes into play if one is interested in \textit{reverse} inequalities of Faber-Krahn type. 
Namely, in particular cases, the first eigenvalue of a doubly connected domain $\Omega$ can be compared to the first eigenvalue of a certain \textit{annulus}. 
Several such results are known in the literature and we overview them below, see Table~\ref{tab1} for the summary. 

In the Neumann-Dirichlet and Neumann-positive Robin cases, i.e., $\hi=0$ and $\ho \in (0,+\infty]$, \textsc{Payne \& Weinberger} \cite[Sections~II, III]{payne} obtained the upper bound 
	\begin{equation}\label{eq:RFKRN}
		\la_1^{\mathcal{NR}}(\Om)\leq \la_1^{\mathcal{NR}}({\A_{r,R}}),
	\end{equation}
    where $R>r$ satisfy 
	\begin{equation}\label{eq:RFKRN:as}
	|\Om|=|\A_{r,R}| 
	\quad \text{and} \quad 
	|\pa \Omo|_1=|\pa B_R|_1.
	\end{equation}
    Here, $\Omi$ is allowed to be empty, that is, there is only the Robin boundary condition on $\pa\Om$.   
    In the case $\Omi = \emptyset$ and without assuming that $\Omo$ is simply-connected, the inequality \eqref{eq:RFKRN} was obtained by \textsc{Freitas \& Krej{\v{c}}i{\v{r}}{\'\i}k} in \cite[Theorem~4]{freitas2015first} for the negative Robin parameter. 
    The inequality \eqref{eq:RFKRN} was extended by \textsc{Paoli, Piscitelli, \& Trani} in \cite[Theorem~3.1]{PaoliPiscitelli} to the Robin parameter of \textit{any sign}, the general higher-dimensional case, and the $p$-Laplace operator, assuming that $\Omo$ is convex. 
    We also refer to \cite[Theorem~1.2]{AnoopMrityunjoy} and \cite{cito2024stability} for closely related results. 

In the Dirichlet-Neumann case, i.e., $\hi = +\infty$ and $\ho = 0$, it was proved by \textsc{Hersch} \cite[Section~5]{hersch-contrib} that
    \begin{equation}\label{eq:RFKND}
		\la_1^{\mathcal{DN}}(\Om)\leq \la_1^{\mathcal{DN}}({\A_{r,R}}),
	\end{equation}
     where $R>r$ satisfy 
	\begin{equation}\label{eq:RFKND:as}
	|\Om|=|\A_{r,R}| 
	\quad \text{and} \quad 
	|\pa \Omi|_1=|\pa B_r|_1.
	\end{equation}
    This result was extended by \textsc{Della Pietra \& Piscitelli} in \cite[Theorem~1.1]{DellaPiscitelli} to the positive Robin-Neumann boundary conditions, i.e., $\hi \in (0,+\infty)$ and $\ho = 0$, the general higher-dimensional case, and the $p$-Laplace operator, assuming that $\Omi$ is convex. 
    Note that, in view of the continuity of $\la_1^{\mathcal{RN}}(\Om)$ with respect to $\hi$ (see \cite[Proposition 2.3]{DellaPiscitelli}), the result of \cite[Theorem~1.1]{DellaPiscitelli} remains valid also for the Dirichlet-Neumann boundary conditions. 
    The latter case is also considered in \cite[Theorem~1.6]{AnoopMrityunjoy}. 
    Moreover, we refer to \cite[Theorem~1.1]{ABD2023} for the Robin-Neumann case with \textit{any sign} of the Robin parameter, under the assumption that $\Omi$ is a disk. 
        
In the purely Dirichlet case $\hi,\ho=+\infty$, \textsc{Hersch} \cite[Section~3]{hers} established the inequality 
	\begin{equation}\label{eq:hersch}
		\la_1^{\mathcal{DD}}(\Om)\leq \la_1^{\mathcal{DD}}({\A_{r,R}})
	\end{equation}
    for the class of all $\Omega$ with prescribed values of $|\Omega|$, $|\pa\Omi|_1$, $|\pa\Omo|_1$ satisfying 
    \begin{equation}\label{eq:hersc:as0}
    |\pa\Omo|_1^2 - |\pa\Omi|_1^2 = 4\pi |\Omega|,
    \end{equation}
    and $R>r$ are such that the annulus $\A_{r,R}$ belongs to this class. 
    Equivalently, \eqref{eq:hersch} holds for any $\Omega$ for which there exist $R>r$ such that
	\begin{equation}\label{eq:hersc:as}
	|\Om|=|\A_{r,R}|,
	\quad
	|\pa \Omi|_1=|\pa B_r|_1,
	\quad
	|\pa \Omo|_1=|\pa B_R|_1.
 	\end{equation}
    Let us note that the arguments of \textsc{Hersch} are based on a result of  \textsc{Weinberger} \cite{weineffect}, but its usage lacks rigor in one important aspect, see a discussion after the statement of Theorem~\ref{thm:RR} below.   
    Recently, the inequality \eqref{eq:hersch} has been claimed by \textsc{Gavitone \& Piscitelli} in \cite[Theorem~1.2]{GP} for the Dirichlet-positive Robin boundary conditions, i.e., $\hi=+\infty$ and $\ho \in (0,+\infty)$, in the general higher-dimensional case, and assuming that $\Omi$, $\Omo$ are convex. 
    However, the constructed test function in the proof of \cite[Theorem~1.2]{GP} can be essentially discontinuous and hence it does not belong to $\tilde{H}^1(\Omega)$, in general. 
    It is also worth mentioning that the inequality of the type \eqref{eq:hersch} was obtained by \textsc{Krej{\v{c}}i{\v{r}}{\'\i}k \& Lotoreichik} \cite{krejvcivrik2024optimisation,krejvcivrik2020optimisation2} in the exterior domain case, that is, $\Omo = \mathbb{R}^2$ and $\Omi$ is bounded, under the negative Robin boundary condition on $\pa\Omi$. 
    
\begin{center}
\begin{table}[ht]
\begin{tabular}{ |c|| c| c| c | c|}
\hline
 \diagbox[]{$\hi$}{$\ho$} & \makecell{$\ho<0$ \\ (negative Robin)} & \makecell{$\ho=0$ \\ (Neumann)}& \makecell{$\ho>0$ \\ (positive Robin)} & \makecell{$\ho=+\infty$\\ (Dirichlet)} \\ 
 \hline
 \hline
 \makecell{$\hi<0$ \\ (negative Robin)} & Theorem~\ref{thm:RR} & \makecell{\cite[Theorem~1.1]{ABD2023}\\ ($\Omi$ disk),\\ Theorem~\ref{thm:NR1}} & ? & \makecell{\cite{krejvcivrik2024optimisation,krejvcivrik2020optimisation2}\\ ($\Omo = \mathbb{R}^2$ and \\ $\Omi$ bounded)}\\  
 \hline
 \makecell{$\hi=0$ \\ (Neumann)} & \makecell{\cite[Theorem 3.1]{PaoliPiscitelli}\\ ($\Omo$ convex), \\ \cite[Theorem~4]{freitas2015first} \\ ($\Omi$ empty), \\Theorem~\ref{thm:RN1}} & trivial & \cite[Section III]{payne} & \cite[Section II]{payne}\\
 \hline
 \makecell{$\hi>0$ \\ (positive Robin)} & ? & \makecell{\cite[Theorem 1.1]{DellaPiscitelli}\\($\Omi$ convex), \\Theorem~\ref{thm:NR1}} & Theorem~\ref{thm:RR} & Theorem~\ref{thm:RR} \\
 \hline
 \makecell{$\hi=+\infty$\\ (Dirichlet)} & ? & 
 \cite[Section~5]{hersch-contrib} & 
 Theorem~\ref{thm:RR} & \makecell{\cite[Section~3]{hers}, \\ Theorem~\ref{thm:RR}} \\
 \hline
\end{tabular}
\caption{Table of results on the inequality $\lambda_1(\Omega) \leq \lambda_1(\A_{r,R})$ in the planar case. 
We recall that $\Omo$ is bounded and $\Omi$ is nonempty, unless otherwise explicitly mentioned. The question mark indicates that the corresponding inequality is unknown (to the best of our knowledge).} 
\label{tab1}
\end{table}
\end{center}

\subsection{Main result}
The above overview indicates that the available results on 
the reverse Faber-Krahn inequalities (RFK inequalities, for short) of the type \eqref{eq:RFKRN}, \eqref{eq:RFKND}, \eqref{eq:hersch}  are scattered in the literature, cover only particular cases of parameters, sometimes impose geometric restrictions on $\Omega$, and some of the arguments contain imprecisions. 

The aim of the present work is to provide a unified result with respect to all possible values of the Robin parameters $\hi,\ho$ such that $\hi \cdot \ho \in [0,+\infty]$.
\begin{theorem}[RFK inequality]\label{thm:RR}
    Let $\hi,\ho\in \R\cup \{+\infty\}$, $(\hi,\ho) \neq (0,0)$, and $\hi \cdot \ho \geq 0$.  
    Let $R>r>0$ be such that  
    \begin{equation}\label{eq:thm:RR:as}
    |\Om|=|\A_{r,R}|,
    \quad 
    \hi \, |\pa \Omi|_1 = \hi \, |\pa B_r|_1,
    \quad
    \ho \, |\pa \Omo|_1 = \ho \, |\pa B_R|_1.
    \end{equation}
    Then 
    \begin{equation}
        \la_1^{\mathcal{RR}}(\Om)\leq \la_1^{\mathcal{RR}}(\A_{r,R}).
    \end{equation}
\end{theorem}

\begin{remark}
In \eqref{eq:thm:RR:as}, the multiplication by $\hi$ and $\ho$ effectively means that, in the case $\hi \cdot \ho = 0$, the length of the Neumann boundary component is not constrained, leading to either \eqref{eq:RFKRN:as} or \eqref{eq:RFKND:as}. 
In the case $\hi \cdot \ho > 0$, the assumption \eqref{eq:thm:RR:as} is equivalent to \eqref{eq:hersc:as}. 
\end{remark}

The proof of Theorem~\ref{thm:RR} is decomposed in two cases: $\hi \cdot \ho = 0$ and $\hi \cdot \ho > 0$.

In the case $\hi \cdot \ho = 0$, we state and prove the corresponding RFK inequalities separately in Theorems~\ref{thm:NR1} and \ref{thm:RN1}, under  slightly more general assumptions on $\Omega$. 
These theorems also serve as important ingredients in the proof of Theorem~\ref{thm:RR} in the case $\hi \cdot \ho > 0$. 
Our arguments for Theorems~\ref{thm:NR1} and \ref{thm:RN1} are based on the method of interior parallels and inspired by \cite{hers,payne}. 
This approach provides an elegant way to construct a test function for the variational characterization \eqref{eq:lambda1} of $\lambda_1^{\mathcal{RR}}(\Omega)$
whose level sets are essentially the parallel sets to the boundary (see Section~\ref{sec:proof_RN} for more details). 

    In the case $\hi \cdot \ho > 0$, we have either two negative Robin boundary conditions, or two positive Robin boundary conditions, or mixed positive Robin and Dirichlet boundary conditions, or the purely Dirichlet boundary condition. 
    We recall that in the latter case the result of Theorem~\ref{thm:RR} is essentially due to \textsc{Hersch} \cite{hers} and \textsc{Weinberger} \cite{weineffect}. 
    More precisely, for the first eigenfunction $u$, \textsc{Weinberger} introduced a set $\Gi \subset \Omega$ such that each trajectory of the gradient flow of $u$ (a \textit{flow line}) starting in $\Gi$ terminates on $\pa\Omi$ in finite time (see Section~\ref{sec:proof-RR} for details). 
    It is then shown that $\Gi$ is open, the part of its boundary $\widetilde{\gamma} = \pa \Gi \cap \Omega$ (called the \textit{effectless cut}) can be represented as a finite union of flow lines, arcs of critical points, and isolated critical points, and $\pa u/\pa \nu = 0$ holds almost everywhere on $\widetilde{\gamma}$. 
    Here we note that we find some of the arguments in \cite{weineffect} lacking details. 
    In particular, it is tempting to think that, after removing isolated points from $\widetilde{\gamma}$ (which was done in \cite{weineffect}), $\widetilde{\gamma}$ becomes a Jordan curve. 
   However, even the connectedness of $\widetilde{\gamma}$ is not immediate.

    If we \textit{assume} that $\widetilde{\gamma}$ is a sufficiently regular Jordan curve surrounding the inner boundary $\pa\Omi$, then the restriction of the first eigenfunction $u$ to the disjoint subdomains $\Gi$ and $\Go := \Omega\setminus \overline{\Gi}$ becomes the first eigenfunction with the Dirichlet-Neumann and Neumann-Dirichlet boundary conditions, respectively, where the Neumann boundary condition is placed on $\widetilde{\gamma}$ and the Dirichlet boundary condition is placed on $\pa\Omi$, $\pa\Omo$.
    Moreover, these first eigenfunctions correspond to the eigenvalues $\lambda_1^{\mathcal{DN}}(\Gi)$ and $\lambda_1^{\mathcal{ND}}(\Go)$ satisfying
    \begin{equation}\label{eq:llleq}
    \lambda_1^\mathcal{DD}(\Omega)
    =
    \lambda_1^{\mathcal{DN}}(\Gi)
    \quad \text{and} \quad 
    \lambda_1^\mathcal{DD}(\Omega)
    =
    \lambda_1^{\mathcal{ND}}(\Go).
    \end{equation}
    Then, by \cite[Section II]{payne} and \cite[Section~5]{hersch-contrib} (see also Theorem~\ref{thm:NR1}, \ref{thm:RN1}), a number $\sigma \in (r,R)$ can be found such that
    \begin{equation}\label{eq:llleq2}
	\lambda_1^{\mathcal{DN}}(\Gi)
	\leq 
	\lambda_1^{\mathcal{DN}}(\A_{r,\sigma})
    \quad \text{and} \quad 
    \lambda_1^{\mathcal{ND}}(\Go) 
    \leq 
    \lambda_1^{\mathcal{ND}}(\A_{\sigma,R}).
    \end{equation}
    Using an appropriate characterization of $\lambda_1^\mathcal{DD}(\A_{\sigma,R})$ (cf.\ Lemma~\ref{lem:character}), a combination of \eqref{eq:llleq} and \eqref{eq:llleq2} eventually leads to the desired upper bound $\la_1^{\mathcal{DD}}(\Om)\leq \la_1^{\mathcal{DD}}({\A_{r,R}})$ under the constraints \eqref{eq:hersc:as}, see \cite{hers}. 
 
    However, in general, $\widetilde{\gamma}$ might contain cusps and ``cracks'' (that is, segments that have $\Gi$ on both sides). As a result, $\widetilde{\gamma}$ might lack the regularity required to directly perform the integration by parts over $\Gi$ and to ensure the compactness of the embedding of $\tilde{H}^1(\Gi)$  to $L^2(\Gi)$ in a straightforward manner, as the standard ``rooms and passages'' example indicates (see, e.g., \cite{fraenkel1979regularity}). 
    In fact, in a very close context, \textsc{Band, Cox, \& Egger} \cite{band2020defining} proved related results whenever $u$ is a Morse function. But in the present settings we do not assume the non-degeneracy of critical points and impose more general boundary conditions, and hence we cannot directly use the results from \cite{band2020defining} in our analysis. 
    Consequently, the validity of the equalities \eqref{eq:llleq}, although anticipated, is far from an obvious claim. 
    
    In our proof of Theorem~\ref{thm:RR}, we revise the arguments from \cite{hers,weineffect}.
    To avoid possible cracks and isolated critical points on $\widetilde{\gamma}$, we propose to consider instead of $\widetilde{\gamma}$ the boundary of a more regular set $\text{Int}(\overline{\Gi})$, see \eqref{eq:g-star}, \eqref{eq:eff-star}. 
    Moreover, we provide a rigorous justification of the equalities \eqref{eq:llleq}. 
    In this way, we use the result of \cite{kurdyka2000proof} on the validity of the Arnold-Thom conjecture (see Remark~\ref{rem:thom-conjecture}), a density result of \cite{smith1994smooth}, and the Picone identity. 

The rest of the article is organized as follows.  
In Section~\ref{sec:aux}, we provide several auxiliary results on the problem \eqref{cutproblem2d} in annular domains important for our analysis.  
In Section~\ref{section:NRRN}, we formulate and prove Theorems~\ref{thm:NR1} and \ref{thm:RN1} on the RFK inequality for the first Robin-Neumann and Neumann-Robin eigenvalues, respectively. 
Section~\ref{sec:proof-RR} contains the proof of Theorem~\ref{thm:RR} in the Robin-Robin case. Section~\ref{sec:remarks} concludes the article with a few remarks.

\section{Auxiliary facts for annuli}\label{sec:aux}

In this section, we provide a few auxiliary results on the first eigenvalues $\la_1^{\mathcal{RN}}(\A_{r,R})$, $\la_1^{\mathcal{NR}}(\A_{r,R})$, $\la_1^{\mathcal{RR}}(\A_{r,R})$ and the corresponding eigenfunctions. First, we observe that the radial solution of \eqref{cutproblem2d} on $\A_{r,R}$ satisfies the following Sturm-Liouville problem:
\begin{equation}\label{eq:ode}
\left\{
\begin{aligned}
    -(tv'(t))' &= \la t v(t), \quad t\in (r,R),\\
   -v'(r)+\hi v(r) &=0,\\
     v'(R)+\ho v(R) &=0.
\end{aligned}
\right.
\end{equation}

We start with the symmetry and monotonicity of the  first eigenfunctions of the Robin-Neumann problem, cf.\ \cite[Proposition~2.2]{DellaPiscitelli} in the case of the positive Robin parameter. 
\begin{lemma}\label{lem:mono_NR1}
Let $0\neq \hi \in \mathbb{R} \cup \{+\infty\}$. Also, let $R>r>0$ and $u$ be a positive eigenfunction associated with $\la_1^{\mathcal{RN}}(\A_{r,R})$ 
Then $u$ is radially symmetric, i.e., $u(x) = v(|x|)$. 
Moreover,
\begin{enumerate}[label={\rm(\roman*)}]
    \item\label{lem:mono_NR1:1} 
    if $\hi> 0$, then $v'>0$ in $(r,R)$, i.e., $u$ is strictly increasing (radially) in $\A_{r,R}$.
    \item\label{lem:mono_NR1:2} 
    if $\hi< 0$, then $v'<0$ in $(r,R)$, i.e., $u$ is strictly decreasing (radially) in $\A_{r,R}$.
\end{enumerate}
Furthermore, if $\hi<+\infty$, then $u> 0$ in $\overline{\A_{r,R}}$. 
\end{lemma}
\begin{proof}
The radiality of $u$ is an immediate consequence of the simplicity of $\la_1^{\mathcal{RN}}(\A_{r,R})$. By taking $v(|x|)=u(x)$, as $\ho=0$, from \eqref{eq:ode} we get 
\begin{equation}\label{eq:NR:ode}
\left\{
\begin{aligned}
    -(tv'(t))' &=\la_1^{\mathcal{RN}}(\A_{r,R}) t v(t), \quad t\in (r,R),\\
    -v'(r)+\hi v(r) &=0,\\
    v'(R)  &=0.
\end{aligned}
\right.
\end{equation}
We proceed to prove only the assertion \ref{lem:mono_NR1:1}; 
the proof of \ref{lem:mono_NR1:2} is analogous.
Let $\hi> 0$. 
Since $\la_1^{\mathcal{RN}}(\A_{r,R})>0$ (see Remark~\ref{rem:sign}) and $v>0$, we deduce from \eqref{eq:NR:ode} that $t\mapsto tv'(t)$ is strictly decreasing in $(r,R)$. 
In view of the boundary condition $v'(R)=0$, we have $v'>0$ in $(r,R)$, and hence the claim follows.

Finally, if $\hi<+\infty$ and $v(t)=0$ for $t=r$ or $t=R$, then we also get $v'(t)=0$ for such $t$ by the boundary conditions. But this contradicts Hopf's lemma. 
Thus, we have $v>0$ in $[r,R]$. 
\end{proof}

Analogous properties are also valid for the positive first eigenfunctions of the Neumann-Robin problem and can be proved in much the same way, see \cite[Proposition~2.5]{PaoliPiscitelli}.
\begin{lemma}\label{lem:mono_RN1}
Let $0\neq\ho \in \mathbb{R} \cup \{+\infty\}$. 
Also, let $R>r>0$ and $u$ be a positive eigenfunction associated with $\la_1^{\mathcal{NR}}(\A_{r,R})$. 
Then $u$ is radially symmetric, i.e., $u(x) = v(|x|)$.
Moreover, 
\begin{enumerate}[label={\rm(\roman*)}]
    \item\label{lem:mono_RN1:1} 
    if $\ho> 0$, then $v'<0$ in $(r,R)$, i.e., $u$ is strictly decreasing (radially) in $\A_{r,R}$.
    \item\label{lem:mono_RN1:2} 
    if $\ho< 0$, then $v'>0$ in $(r,R)$, i.e., $u$ is strictly increasing (radially) in $\A_{r,R}$.
    \end{enumerate}
Furthermore, if $\ho<+\infty$, then $u> 0$ in $\overline{\A_{r,R}}$. 
\end{lemma}

Let us now give a counterpart of the previous two lemmas in the Robin-Robin case. 
\begin{lemma}\label{lem:mono_RR1}
Let $\hi,\ho \in \mathbb{R} \cup \{+\infty\}$ be such that $\hi \cdot \ho > 0$. 
Also, let $R>r>0$ and $u$ be a positive eigenfunction associated with $\la_1^{\mathcal{RR}}(\A_{r,R})$. 
Then $u$ is radially symmetric, i.e., $u(x) = v(|x|)$.
Moreover, there exists $\sigma \in (r,R)$ such that 
\begin{enumerate}[label={\rm(\roman*)}]
    \item\label{lem:mono_RR1:1}
    if $\hi,\ho> 0$, then $v'>0$ in $(r,\sigma)$, $v'(\sigma)=0$, and $v'<0$ in $(\sigma,R)$.
    \item\label{lem:mono_RR1:2} 
    if $\hi,\ho< 0$, then $v'<0$ in $(r,\sigma)$, $v'(\sigma)=0$, and $v'>0$ in $(\sigma,R)$.
    \end{enumerate}
In particular, $u|_{\A_{\sigma,R}}$ and $u|_{\A_{r,\sigma}}$ are positive eigenfunctions corresponding to $\la_1^{\mathcal{NR}}(\A_{\sigma,R})$ and $\la_1^{\mathcal{RN}}(\A_{r,\sigma})$, respectively. 
\end{lemma}
\begin{proof}
    The radiality of $u$ follows from the simplicity of $\la_1^{\mathcal{RR}}(\A_{r,R})$. 
    We proceed to prove only the assertion \ref{lem:mono_RR1:1}; the proof of \ref{lem:mono_RR1:2} is analogous.
    Let $\hi,\ho> 0$. 
    From \eqref{eq:ode}, we first observe that $v'(r) > 0$ and $v'(R) < 0$.
    If we suppose that $v'(r)= 0$, then, from the boundary condition we get $v(r)=0$, which contradicts Hopf's lemma asserting that $v'(r)>0$.
    In the same way, it can be shown that $v'(R)<0$. 
    Consequently, there exists $\si\in(r,R)$ such that $v(\si)=\max_{r < t < R} v(t)$.
    In particular, we have $v'(\si)=0$, and hence $v$ is a solution of the following problem in $(r,\si)$:
    \begin{equation}\label{eq:SLRN}
\left\{
\begin{aligned}
    -(tv'(t))' &=\la_1^{\mathcal{RR}}(\A_{r,R}) t v(t), \quad t\in (r,\si),\\
    -v'(r)+\hi v(r) &=0,\\
    v'(\si)  &=0.
\end{aligned}
\right.
\end{equation}
Since $v$ is positive in $(r,\si)$, we see that $v$ is a first eigenfunction of \eqref{eq:SLRN}. Therefore, Lemma~\ref{lem:mono_NR1}-\ref{lem:mono_NR1:1} gives $v'>0$ in $(r,\si)$. 
A similar analysis on $(\si,R)$ yields $v'<0$ in $(\si,R)$.
\end{proof}

Let us now investigate the domain monotonicity of the first Robin-Neumann and Neumann-Robin eigenvalues.
\begin{lemma}\label{lem:domain-monot}
    Let $R>r>0$. 
    Then the following assertions hold:
    \begin{enumerate}[label={\rm(\roman*)}]
        \item\label{lem:domain-monot:1} 
        if $\hi>0$, then $\delta \mapsto \lambda_1^{\mathcal{RN}}(\A_{r,\delta})$ is strictly decreasing in $(r,R)$, and 
        if $\ho>0$, then $\delta \mapsto \lambda_1^{\mathcal{NR}}(\A_{\delta,R})$ is strictly increasing in $(r,R)$.
\item\label{lem:domain-monot:2} 
if $\hi<0$, then
$\delta \mapsto \lambda_1^{\mathcal{RN}}(\A_{r,\delta})$ is strictly increasing in $(r,R)$, and 
if $\ho < 0$, then $\delta \mapsto \lambda_1^{\mathcal{NR}}(\A_{\delta,R})$ is strictly decreasing in $(r,R)$. 
\end{enumerate}
\end{lemma}
\begin{proof}
\ref{lem:domain-monot:1} 
Let $\ho> 0$. 
For any $r<\delta_1<\delta_2<R$, taking a first eigenfunction $v(|x|)$ corresponding to $\lambda_1^{\mathcal{NR}}(\A_{\delta_2,R})$ and extending it by the constant $v(\delta_2)$ to $\A_{\delta_1,R}$, we get a test function for $\lambda_1^{\mathcal{NR}}(\A_{\delta_1,R})$ whose Rayleigh quotient is strictly smaller than $\lambda_1^{\mathcal{NR}}(\A_{\delta_2,R})$, since $\lambda_1^{\mathcal{NR}}(\A_{\delta_2,R})>0$ by Remark~\ref{rem:sign}. If $\hi>0,$ a similar extension argument ensures the strict monotonicity of $\delta \mapsto \lambda_1^{\mathcal{RN}}(\A_{r,\delta})$. 

\ref{lem:domain-monot:2} 
Let $\hi < 0$ and fix $r<\delta_1<\delta_2<R$. 
Let $v(|x|)$ be a positive eigenfunction  corresponding to $\lambda_1^{\mathcal{RN}}(\A_{r,\delta_2})$, so that $v$ satisfies 
\begin{equation}\label{eq:NR:ode-delta}
	\left\{
	\begin{aligned}
		-(tv'(t))' &=\la_1^{\mathcal{RN}}(\A_{r,\delta_2}) t v(t), \quad t\in (r,\delta_2),\\
		-v'(r)+\hi v(r) &=0,\\
		v'(\delta_2)  &=0.
	\end{aligned}
	\right.
\end{equation}
Multiplying the equation in \eqref{eq:NR:ode-delta} by $v$ and integrating by parts over a smaller interval $(r,\delta_1)$, we get
$$
\int_r^{\delta_1} t (v'(t))^2 \dt
+
\hi r (v(r))^2
-
\delta_1 v'(\delta_1) v(\delta_1)
=
\la_1^{\mathcal{RN}}(\A_{r,\delta_2}) 
\int_r^{\delta_1}
t (v(t))^2 \dt.
$$
In view of Lemma~\ref{lem:mono_NR1}~\ref{lem:mono_NR1:2}, we have $v'(\delta_1)<0$, which yields
$$
\int_r^{\delta_1} t (v'(t))^2 \dt
+
\hi r (v(r))^2
<
\la_1^{\mathcal{RN}}(\A_{r,\delta_2}) 
\int_r^{\delta_1} t (v(t))^2 \dt.
$$
Observing that the restriction of $v(|\cdot|)$ to $\A_{r,\delta_1}$ is a valid test function for the definition of $\la_1^{\mathcal{RN}}(\A_{r,\delta_1})$, we arrive at the desired strict monotonicity:
$$
\la_1^{\mathcal{RN}}(\A_{r,\delta_1})
\leq
\frac{\int_r^{\delta_1} t (v'(t))^2 \dt
	+
	\hi r (v(r))^2}{\int_r^{\delta_1}	t (v(t))^2 \dt}
< 
\la_1^{\mathcal{RN}}(\A_{r,\delta_2}).
$$
The proof for the case of $\delta \mapsto \lambda_1^{\mathcal{NR}}(\A_{\delta,R})$ is analogous. 
\end{proof}

Finally, we provide the following minimax characterizations of $\lambda_1^{\mathcal{RR}}(\A_{r,R})$, 
see \cite[Lemma~2.2]{bobkov2017some} and \cite[Theorem~3.3]{kajikiya2024asymptotic} for related statements.
\begin{lemma}\label{lem:character}
Let $\hi,\ho \in \mathbb{R} \cup \{+\infty\}$ and $\hi \cdot \ho > 0$, and let $R>r>0$.
Then
\begin{align}
\lambda_1^{\mathcal{RR}}(\A_{r,R}) 
&=
\min_{\delta \in (r,R)}
\max \{
\lambda_1^{\mathcal{RN}}(\A_{r,\delta}),
\lambda_1^{\mathcal{NR}}(\A_{\delta,R})
\}
\\
&=
\max_{\delta \in (r,R)}
\min \{
\lambda_1^{\mathcal{RN}}(\A_{r,\delta}),
\lambda_1^{\mathcal{NR}}(\A_{\delta,R})
\}.
\end{align}
\end{lemma}
\begin{proof}
In view of Lemma~\ref{lem:mono_RR1}, there exists $\sigma \in (r,R)$ such that 
$$
\lambda_1^{\mathcal{RR}}(\A_{r,R} )
=
\lambda_1^{\mathcal{RN}}(\A_{r,\sigma})
=
\lambda_1^{\mathcal{NR}}(\A_{\sigma,R}).
$$
The strict monotonicity of  $\lambda_1^{\mathcal{RN}}(\A_{r,\delta})$ and $\lambda_1^{\mathcal{NR}}(\A_{\delta,R})$ with respect to $\delta$ described in Lemma~\ref{lem:domain-monot} implies that $\sigma$ is the unique point such that $\lambda_1^{\mathcal{RN}}(\A_{r,\sigma})=\lambda_1^{\mathcal{NR}}(\A_{\sigma,R})$.
Consequently, we have
\begin{gather}
\min_{\delta \in (r,R)}
\max \{
\lambda_1^{\mathcal{RN}}(\A_{r,\delta}),
\lambda_1^{\mathcal{NR}}(\A_{\delta,R})
\}
\\
=
\lambda_1^{\mathcal{RN}}(\A_{\sigma, r})
=
\lambda_1^{\mathcal{NR}}(\A_{\sigma,R})
\\
=
\max_{\delta \in (r,R)}
\min \{
\lambda_1^{\mathcal{RN}}(\A_{r,\delta}),
\lambda_1^{\mathcal{NR}}(\A_{\delta,R})
\},
\end{gather} 
which gives the desired conclusion. 
\end{proof}

\section{Robin-Neumann and Neumann-Robin cases}\label{section:NRRN}

In this section, we derive the RFK inequality for the first Robin-Neumann and Neumann-Robin eigenvalues without any sign assumption on the Robin parameter and under slightly more general regularity assumptions on $\Omega$.
We start with the result for the first eigenvalue $\la_1^{\mathcal{RN}}(\Om)$ of the problem
\begin{align}\tag{$\mathcal{RN}$}\label{eq:Neumann_Robin1}
\left\{\begin{aligned}
    -\Delta u &= \lambda u \quad\text{in} \quad\Omega,\\
    \frac{\pa u}{\pa \nu}+ \hi u &=0 \quad\;\;\text{on}\quad  \pa \Omi,\\
     \frac{\pa u}{\pa \nu}&=0 \quad\;\;\text{on}\quad  \pa \Omo.
	\end{aligned}
	 \right.
\end{align}

\begin{theorem}[RFK inequality in the Robin-Neumann case]\label{thm:NR1}
Let $\Omi,\Omo$ be Lipschitz. 
Also, let $0 \neq \hi \in \R\cup \{+\infty\}$ and $R>r>0$ be such that 
$$
|\Om|=|\A_{r,R}|
\quad \text{and} \quad 
|\pa \Omi|_1=|\pa B_{r}|_1.
$$
Then 
\begin{equation}
    \la_1^{\mathcal{RN}}(\Om)\leq \la_1^{\mathcal{RN}}(\A_{r,R}).
\end{equation}
\end{theorem}

Now we state the corresponding result for the first eigenvalue $\la_1^{\mathcal{NR}}(\Om)$ of the problem
\begin{align}\tag{$\mathcal{NR}$}\label{eq:Robin_Neumann1}
\left\{\begin{aligned}
    -\Delta u &= \lambda u \quad\text{in} \quad\Omega,\\
    \frac{\pa u}{\pa \nu}&=0 \quad\;\;\text{on}\quad  \pa \Omi,\\
    \frac{\pa u}{\pa \nu}+ \ho u&=0 \quad\;\;\text{on}\quad  \pa \Omo.
	\end{aligned}
	 \right.
\end{align}
\begin{theorem}[RFK inequality in the Neumann-Robin case]\label{thm:RN1}
Let $\Omi,\Omo$ be Lipschitz, where $\Omi$ is allowed to be empty.
Also, let $0 \neq \ho \in \R\cup \{+\infty\}$ and $R>r \geq 0$ be such that 
$$
|\Om|=|\A_{r,R}|
\quad \text{and} \quad 
|\pa \Omo|_1=|\pa B_{R}|_1.
$$
Then 
\begin{equation}\label{Bound_NR}
    \la_1^{\mathcal{NR}}(\Om)\leq \la_1^{\mathcal{NR}}(\A_{r,R}).
\end{equation}
Here, in the case $r=0$, we have $\Omega = \A_{0,R} = B_R$ and equality holds in \eqref{Bound_NR}.
\end{theorem}

\subsection{Preliminaries for the method of interior parallels}\label{sec:proof_RN}

Hereinafter, we denote by $d(x, K)$ the standard Euclidean distance from a point $x$ to a set $K$. 
\begin{definition}[Parallel sets]\label{def:parallel}
    Let $\de>0$ and $D\subset \R^2$ be a bounded domain. 
    Denote 
        $$
        D_{-\de}:=\{x\in D:~ d(x,\pa D)\geq\de\}
        ~\text{and}~
        D_{\de}:=\{ x\in D^c:~ d(x,\pa D)\geq \de \}^c.
        $$
        The sets $\pa D_{-\de}$ and $\pa D_{\de}$ are called  \textit{inner parallel set} and \textit{outer parallel set} of $D$ at a distance $\de$, respectively.
\end{definition}

Let us recall some classical estimates on the parallel sets, see \cite{Nagy, Makai1959}.
\begin{lemma}[Nagy's inequalities]\label{lem:Nagy}
    Let $D \subset \R^2$ be a bounded, simply connected domain and $D^\#$ be an open disk centered at the origin and satisfying $|\pa D|_1 = |\pa D^\#|_1$. Then the following estimates hold:
    \begin{enumerate}[label={\rm(\roman*)}]
        \item\label{lem:Nagy:1} $|\pa D_{-\de}|_1 \leq  |\pa D|_1-2\pi\de=|\pa D^\#_{-\de}|_1$ for all $\de\in (0,r_D)$, where $r_D$ is the inradius of $D$.  
        \item\label{lem:Nagy:2} $|\pa D_{\de}|_1 \leq  |\pa D|_1+2\pi\de=|\pa D^\#_\de|_1$ for all $\de>0$.
    \end{enumerate}
\end{lemma}
\begin{remark}\label{rem:Nagy}
    By the Steiner formula (cf.\ \cite[Chapter~4]{Schneider}), for any convex domain $D$ we have
    $$
    |\pa D_{\de}|_1 =  |\pa D|_1 + 2\pi\de=|\pa D^\#_\de|_1,
    $$
    which is the equality case in Lemma \ref{lem:Nagy}-\ref{lem:Nagy:2}.
    This shows that Nagy's inequality is not rigid, in general.
\end{remark}

\subsubsection{Parametrizations for the Robin-Neumann problem} 
Assume as in Theorem~\ref{thm:NR1} that $R>r>0$ are such that
$$
|\Om|=|\A_{r,R}| 
\quad \text{and}\quad 
|\pa \Omi|_1=|\pa B_{r}|_1.
$$ 
Introduce the following notation:
\begin{align}
    \label{eq:delta*}
    \delta_* & =\sup\{\delta>0:~ |\pa (\Omi)_\de\cap \Omega|_1 >0 \},\\
    \label{eq:s1}
    s(\delta) & =
   |\pa (\Omi)_\de\cap \Omega|_1  \quad \text{for}~\delta > 0,\\
    \label{eq:S1}
    S(\delta) &=|\pa (B_{r})_\de|_1 = 2\pi (r+\delta) \quad \text{for}~\de >0.
\end{align}
By applying Lemma \ref{lem:Nagy}-\ref{lem:Nagy:2} to $\Omi$, we directly obtain the following result. 
\begin{corollary} \label{cor:Nagy1}
We have
$$
s(\de)\leq |\pa (\Omi)_\de|_1\leq S(\delta) \quad \text{for all } 
\de >0.
$$
\end{corollary}

Consider the following parametrizations motivated by \textsc{Hersch} \cite{hers} (see also \cite{AnoopMrityunjoy,AnoopAshok}):
\begin{equation}\label{lL_delta}
t(\delta)=\int_{0}^{\delta}\frac{1}{s(\rho)}\drh
~\text{ for}~ \de\in (0,\delta^*)
\quad\text{and}\quad
T(\delta)=\int_{0}^{\delta}\frac{1}{S(\rho)}\drh
~\text{ for}~ \de\in (0,R-r].
\end{equation}
Since the functions $s$ and $S$ are positive in $(0,\delta^*)$ and $(0,+\infty)$, respectively, $t$ and $T$ are strictly increasing. 
It is also clear that $T \in C^1(0,R-r)$ and $t$ is absolutely continuous in $(0,\de^*)$, and hence $t'$ is at least in $L^1_{\text{loc}}(0,\de^*)$. 
Let us denote 
$$
t_*
=
\lim_{\delta \to \delta_*}
t(\delta) 
=
\sup_{\de \in (0,\de_*)} t(\de)
\quad \text{and} \quad 
T_{\#}=T(R-r) = \max_{\de \in (0,R-r]}T(\de), 
$$ 
and define
\begin{align}
    \label{eq:g}
    g(\alpha) &=s(t^{-1}(\alpha)) \quad\text{for}~\alpha \in [0,t_*],\\
    \label{eq:G}
    G(\alpha) &=S(T^{-1}(\alpha))
    \quad \text{for}~\alpha \in [0,T_{\#}].
\end{align}
In view of Corollary \ref{cor:Nagy1}, 
the parametrizations $t$ and $T$ satisfy the following properties, and we refer to \cite{AnoopAshok,AnoopMrityunjoy} for a detailed proof of these statements.
\begin{lemma}\label{lem:paramet}
 The following assertions hold:
    \begin{enumerate}[label={\rm(\roman*)}]
        \item\label{lem:paramet:1} $R-r\leq \de_*$ and $T_{\#}\leq t_*$.        
        \item\label{lem:paramet:2} $g(\alpha)\leq G(\alpha)$ for all $\alpha \in [0,T_\#]$. Moreover,
$$
|\Omega|=\int_{0}^{t_*}g(\alpha)^{2}\da=\int_{0}^{T_\#}G(\alpha)^{2}\da=|\A_{r,R}|.
$$
\end{enumerate}
\end{lemma}

\medskip
\subsubsection{Parametrizations for the Neumann-Robin problem}
Let $\Omega$ and $\A_{r,R}$ be as in Theorem \ref{thm:RN1}, so that
$$
|\Om|=|\A_{r,R}| 
\quad \text{and}\quad 
|\pa \Omo|_1=|\pa B_{R}|_1.
$$
Define
\begin{align*}
    \de^* &=\sup\{\de>0:~ |\pa (\Omo)_{-\de} \cap \Om|_1>0\},\\
    s(\de) &= |\pa (\Omo)_{-\de} \cap \Om|_1 \quad\text{for}~ \de>0,\\
    S(\de) &= |\pa (B_R)_{-\de}|_1 \quad\text{for}~ \de>0.
\end{align*}
Evidently, $S(\de) = 2\pi(R-\de)$ for $\de \in (0,R]$. 
One can observe that $(\Omo)_{-\de}=\emptyset$ for $\de>r_{\Omo}$, where $r_{\Omo}$ is the inradius of $\Omo$, see, e.g., \cite[p.~148]{Schneider}. Therefore, we immediately get $\de^*\leq r_{\Omo}$. 
Using Lemma \ref{lem:Nagy}-\ref{lem:Nagy:1} for $\Omo$, we obtain the following consequence.
\begin{corollary}\label{cor:Nagy2}
    We have 
    $$
    s(\de)\leq |\pa (\Omo)_{-\de}|_1\leq S(\delta) \quad \text{for any}~ \de\in (0,\de^*).
    $$
\end{corollary}

Consider the following parametrizations inspired by \textsc{Payne \& Weinberger} \cite{payne} (see \cite{AnoopMrityunjoy,AnoopAshok} for higher dimensional generalizations):
\begin{equation}\label{tT_delta}
l(\delta)=\displaystyle\int_{0}^{\delta}s(\rho)\drh 
\quad\text{and}\quad
L(\delta)=\int_{0}^{\delta}S(\rho)\drh.
\end{equation}
Geometrically, $l(\de)$ represents the area of a subset of $\Omega$ enclosed between $\pa \Omo$ and $\pa (\Omo)_{-\de}$, and a similar interpretation holds for $L(\de)$. 
The functions $l$ and $L$ are increasing. 
Moreover, applying Corollary~\ref{cor:Nagy2}, one has
$$
l(\de)\leq L(\de) ~\text{ for all}~ 
\de\in (0,\de^*).
$$
From the definitions of $l$ and $L$, it also follows that 
$$
l(\de^*)=|\Om|=|\A_{r,R}|=L(R-r).
$$
Now we define 
\begin{equation}
    h(\al)=s(l^{-1}(\al)) 
    \quad\text{and}\quad 
    H(\al)=S(L^{-1}(\al)) 
    ~\text{ for}~\al\in [0,|\Om|].
\end{equation}
Using Corollary \ref{cor:Nagy2} together with the monotonicity properties of $l$ and $L$ mentioned above, we can establish the following properties of $h$ and $H$, which are analogous to those in Lemma \ref{lem:paramet}, see, e.g., \cite{AnoopAshok,AnoopMrityunjoy}. 
\begin{lemma}\label{lem:paramet1}
The following assertions hold:
    \begin{enumerate}[label={\rm(\roman*)}]
        \item $R-r \leq \de^*$. 
        \item $h(\al)\leq H(\al)$ for all $\al\in [0,|\Om|]$.
    \end{enumerate}
\end{lemma}

\medskip
With the auxiliary results of this section in hand, we proceed with the proof of Theorem~\ref{thm:NR1}. 

\subsection{Proof of Theorem \ref{thm:NR1}}
 The case $\hi=+\infty$ is proved in \cite{hers}, and hence we assume $\hi \in \mathbb{R} \setminus \{0\}$. 
 Let $u$ be a positive first eigenfunction of \eqref{eq:Neumann_Robin1} in $\A_{r,R}$; the normalization of $u$ is irrelevant for our analysis.  
 Since $u$ is radially symmetric and positive in $\overline{\A_{r,R}}$ (see Lemma~\ref{lem:mono_NR1}), we can find a smooth positive function $w: (0,R-r)\rightarrow (0,+\infty)$ satisfying the following relation:
\begin{equation}
    u(x)=w(|x|-r) \equiv w(d(x,\pa B_{r})) ~\text{ for}~ x\in\A_{r,R}. 
\end{equation}
Now using the definition \eqref{lL_delta} of the function $T$, we rewrite $u$ in the following way:
\begin{equation}\label{eq:uphi}
u(x)=(w\circ T^{-1})(T(d(x,\pa B_{r})))=\phi (T(d(x,\pa B_{r}))) ~\text{ for}~ x\in\A_{r,R},
\end{equation}
where $\phi:=w\circ T^{-1}$.

We will investigate the cases $\hi>0$ and $\hi<0$ separately.

\underline{Case $\hi>0$.}~ 
Since $u$ is strictly increasing (see Lemma \ref{lem:mono_NR1}-\ref{lem:mono_NR1:1}) and so are $T$ and $d(\cdot,\pa B_{r})$, we see that $\phi$ is strictly increasing in $[0,T_\#]$ and 
\begin{align}
\phi (0) & =\displaystyle \min_{[0,T_\#]} \phi=\min_{\overline{\A_{r,R}}} u:=u_{\rm{min}},\\
\label{eq:phi:incr}
\phi (T_\#) & =\displaystyle \max_{[0,T_\#]} \phi=\max_{\overline{\A_{r,R}}} u:=u_{\rm{max}}.
\end{align}
With the help of the function $t$ (see \eqref{lL_delta}), we define a function $v$ on $\Om$ as follows:
\begin{equation}\label{eq:v}
 		v(x)= \begin{cases}
 		 \phi(t(d(x, \pa \Omi))),& \mbox{if }~ t(d(x, \pa \Omi)) \in [0,T_\#], \\
 		u_{\rm{max}}, & \mbox{if }~ t(d(x, \pa \Omi)) \in (T_\#,t_*]. 
 		\end{cases}
\end{equation}
Later we will justify that $v \in H^1(\Omega)$. 
From the construction, it is clear that 
\begin{align}
\label{eq:v:min}
v_{\rm{min}} &:=\displaystyle\min_{\overline{\Om}}v=\min_{\pa \Omi} v=u_{\rm{min}},\\
\label{eq:v:max}
v_{\rm{max}} &:=\displaystyle\max_{\overline{\Om}}v=\max_{\pa \Omo} v=u_{\rm{max}}.
\end{align}
Since $\phi$ is smooth, $t$ is absolutely continuous, and $d(\cdot, \pa \Omi)$ is Lipschitz, we conclude that $v$ has a weak gradient and $|\nabla v|$ is measurable. 
Noting that $|\nabla d(\cdot, \partial\Omi)| = 1$ a.e.\ in $\Omega$ and $t'(\delta) = 1/s(\delta)$, we have 
$$
|\nabla v(x)| = |\phi'(t(d(x, \pa \Omi)))| \, \frac{1}{s(d(x, \pa \Omi))}
\quad \text{for}~ t(d(x, \pa \Omi)) \in [0,T_\#],
$$ 
and $|\nabla v(x)| = 0$ for $t(d(x, \pa \Omi)) \in (T_\#,t_*]$.
Therefore, applying the co-area formula, we get
\begin{align}
    \int_{\Om}|\nabla v(x)|^2\dx 
    =
    \int_{\Om}|\nabla v(x)|^2 |\nabla d(x, \partial\Omi)|\dx 
    &=\int_{0}^{\infty}\int_{(\pa(\Omi)_\de)\cap \Om}|\nabla v(x)|^2\dsi(x) \dd\\
    \label{eq:coar1}
    &
    =\int_{0}^{t^{-1}(T_\#)}|\phi'(t(\delta))|^2 \frac{1}{s(\delta)}\dd.
    \quad
 \end{align}
After changing the variable $t(\delta)=\alpha$ so that 
$$
\delta = t^{-1}(\alpha)
\quad \text{and} \quad
\dd = s(t^{-1}(\alpha)) \, \da, 
$$
we get from \eqref{eq:coar1} that
\begin{equation}\label{norm:gradv}
    \int_{\Om}|\nabla v(x)|^2\dx=\int_{0}^{T_\#}|\phi'(\alpha)|^2\da.
\end{equation}
In view of the relation \eqref{eq:uphi}, 
arguing in much the same way as above, we also obtain
\begin{equation}\label{norm:gradu}
    \int_{\A_{r,R}}|\nabla u(x)|^2\dx=\int_{0}^{T_\#}|\phi'(\alpha)|^2\da,
\end{equation}
which yields
\begin{equation}\label{norm:gradugradv}
    \int_{\Om}|\nabla v(x)|^2\dx
    =
    \int_{\A_{r,R}}|\nabla u(x)|^2\dx.
\end{equation}

On the other hand, performing similar calculations, we get 
\begin{align}
\label{eq:v-l2}
    \int_{\Om} v(x)^2\dx & =
    \int_{0}^{T_\#}|\phi(\al)|^2 g(\alpha)^2\da+ u_{\rm{max}}^2\int_{T_\#}^{t_*}g(\alpha)^2\da,\\
\label{eq:v-l2:1}
    \int_{\A_{r,R}}u(x)^2\dx & =\int_{0}^{T_\#}|\phi(\al)|^2 G(\alpha)^2\da,
\end{align}
where $g, G$ are defined in \eqref{eq:g}, \eqref{eq:G}, respectively. 
Observe that by Lemma~\ref{lem:paramet} we have
\begin{equation}\label{eq:v-l2:2}
\int_{\Om} v(x)^2\dx 
\leq 
\int_{0}^{T_\#}|\phi(\al)|^2 G(\alpha)^2\da+ u_{\rm{max}}^2\int_{0}^{t_*}g(\alpha)^2\da
=
\int_{\A_{r,R}} u(x)^2\dx + u_{\rm{max}}^2 |\Omega|,
\end{equation}
which in combination with  \eqref{norm:gradugradv} implies that $v\in H^1(\Om)$.

Further recalling that $\phi \leq u_{\rm{max}}$ in $[0,T_\#]$ (see \eqref{eq:phi:incr}) and using Lemma~\ref{lem:paramet}, we obtain 
\begin{align}
    \int_{\Om} v(x)^2\dx-\int_{\A_{r,R}}u(x)^2\dx & \geq u_{\rm{max}}^2 \int_{0}^{T_\#} \left\{g(\alpha)^2-G(\alpha)^2\right\}\da+u_{\rm{max}}^2\int_{T_\#}^{t_*}g(\alpha)^2\da\\
    & = u_{\rm{max}}^2 \int_{0}^{t_*} g(\alpha)^2\da - u_{\rm{max}}^2 \int_{0}^{T_\#} G(\alpha)^2 \da = 0,\\
    \text{i.e.,}~\int_{\Om} v(x)^2\dx & \geq \int_{\A_{r,R}} u(x)^2\dx.
    \label{estimate:L2}
\end{align}
Using the radial monotonicity of $u$ in $\A_{r,R}$, the property \eqref{eq:v:min} of $v$, and the assumption $|\pa B_{r} |_1=|\pa \Omi |_1$, we deduce that
\begin{equation}\label{estimate:bd}
    \int_{\pa B_{r}} u(x)^2\dsi =u_{\rm{min}}^2|\pa B_{r} |_1= v_{\rm{min}}^2|\pa \Omi |_1
    =
    \int_{\pa \Omi} v(x)^2\dsi.
\end{equation}

Since $v$ is a valid test function for the definition \eqref{eq:lambda1} of $\lambda_1^{\mathcal{RN}}(\Omega)$, 
using \eqref{norm:gradugradv}, \eqref{estimate:L2}, and \eqref{estimate:bd}, 
we get the desired inequality $\lambda_1^{\mathcal{RN}}(\Omega) \leq \lambda_1^{\mathcal{RN}}(\A_{r,R})$ as follows:
$$
\lambda_1^{\mathcal{RN}}(\Omega) 
\leq 
\frac{\int_{\Om}|\nabla v|^2\dx+\hi \int_{\pa \Omi} v^2\dsi}{\int_{\Om} v^2\dx}
\leq
\frac{\int_{\A_{r,R}}|\nabla u|^2\dx+\hi \int_{\pa B_r} u^2\dsi}{\int_{\A_{r,R}} u^2\dx}
=
\lambda_1^{\mathcal{RN}}(\A_{r,R}).
$$

\underline{Case $\hi<0$.}~ 
The proof of this case will be in the same vein as above, with appropriate modifications in the definition of the test function $v$. 
Since $\hi<0$, the positive first eigenfunction $u$ of \eqref{eq:Neumann_Robin1} in $\A_{r,R}$ is strictly radially decreasing  (see Lemma~\ref{lem:mono_NR1}-\ref{lem:mono_NR1:2}), and hence it attains its maximum on the Robin boundary $\pa B_{r}$ and minimum on the Neumann boundary $\pa B_{R}$.  
Therefore, the function
$\phi$ defined in \eqref{eq:uphi} is strictly decreasing in $[0,T_\#]$ and 
\begin{align}
\phi (0) & =\displaystyle \max_{[0,T_\#]} \phi=\max_{\overline{\A_{r,R}}} u:=u_{\rm{max}},\\
\phi (T_\#) & =\displaystyle \min_{[0,T_\#]} \phi=\min_{\overline{\A_{r,R}}} u:=u_{\rm{min}}.
\end{align}
Inspired by these facts, we construct a test function for the variational characterization \eqref{eq:lambda1} of $\lambda_1^{\mathcal{RN}}(\Omega)$ that obeys similar monotonicity in $\Om$ along the parallel sets. 
Precisely, we define the test function $v$ in $\Om$ as 
\begin{equation*}
 		v(x)= \begin{cases}
 		 \phi(t(d(x, \pa \Omi))),& \mbox{if }~ t(d(x, \pa \Omi)) \in [0,T_\#], \\
 		u_{\rm{min}}, & \mbox{if }~ t(d(x, \pa \Omi)) \in (T_\#,t_*]. 
 		\end{cases}
 		\end{equation*}
One can easily verify that 
\begin{align}
v_{\rm{min}} &:=\displaystyle\min_{\overline{\Om}}v=\min_{\pa \Omi} v=u_{\rm{min}},\\
v_{\rm{max}} &:=\displaystyle\max_{\overline{\Om}}v=\max_{\pa \Omo} v =u_{\rm{max}}.
\end{align}
Similarly to the previous case, we obtain
\begin{align}\label{grad_2}
    \int_{\Om}|\nabla v(x)|^2\dx &=\int_{\A_{r,R}}|\nabla u(x)|^2\dx,\\
    \label{boundary_2}
    \int_{\pa B_{r}} u(x)^2\dsi 
    =u_{\rm{max}}^2|\pa B_{r} |_1 & = v_{\rm{max}}^2|\pa \Omi |_1=\int_{\pa \Omi} v(x)^2\dsi.
\end{align}
Recalling that $\phi \geq u_{\rm{min}}$ in $[0,T_\#]$ and using Lemma~\ref{lem:paramet}-\ref{lem:paramet:2}, we derive 
\begin{align}
     \int_{\Om} v(x)^2\dx-\int_{\A_{r,R}}u(x)^2\dx &= \int_{0}^{T_\#}|\phi(\al)|^2 \left\{g(\alpha)^2-G(\al)^2\right\}\da+ u_{\rm{min}}^2\int_{T_\#}^{t_*}g(\alpha)^2\da\\
     & \leq u_{\rm{min}}^2 \left\{\int_{0}^{t_*} g(\alpha)^2\da -\int_{0}^{T_\#} G(\alpha)^2 \da\right\}=0,\\
     \text{i.e.,}~ \int_{\Om} v(x)^2\dx & \leq \int_{\A_{r,R}}u(x)^2\dx.\label{L2}
\end{align}

Using the fact that $\la_1^{\mathcal{RN}}(\A_{r,R})<0$ when $\hi<0$ (see Remark~\ref{rem:sign}) and the expressions \eqref{grad_2}, \eqref{boundary_2}, \eqref{L2},  we obtain the desired inequality as follows:
$$
\lambda_1^{\mathcal{RN}}(\Omega) 
\leq 
\frac{\int_{\Om}|\nabla v|^2\dx+\hi \int_{\pa \Omi} v^2\dsi}{\int_{\Om} v^2\dx}
\leq
\frac{\int_{\A_{r,R}}|\nabla u|^2\dx+\hi \int_{\pa B_r} u^2\dsi}{\int_{\A_{r,R}} u^2\dx}
=
\lambda_1^{\mathcal{RN}}(\A_{r,R}).
$$
The proof of Theorem~\ref{thm:NR1} is complete. 
\qed

\medskip
Next, we prove of Theorem \ref{thm:RN1}. The scheme of the arguments closely follows that of Theorem~\ref{thm:NR1} and, therefore, we only provide an outline here.

\subsection{Proof of Theorem \ref{thm:RN1}} 
As before, the case $\ho=+\infty$ is proved in \cite{payne}, and so we restrict our attention to the remaining case $\ho\in \R\setminus \{0\}$. Let $u$ be a  positive first eigenfunction of \eqref{eq:Robin_Neumann1} in $\A_{r,R}$. According to Lemma \ref{lem:mono_RN1}, $u$ is radially symmetric and positive throughout $\overline{\A_{r,R}}$. 
Hence, we define a smooth positive function $w:(0,R-r)\rightarrow (0,+\infty)$ such that 
$$
u(x)=w(R-|x|)=w(d(x,\pa B_R))~\text{ for}~x\in \A_{r,R}.
$$
Using the definition \eqref{tT_delta} of the function $L$, one can rewrite the above relation as 
\begin{equation}
    \label{eq:uphi-2}
u(x)=(w\circ L^{-1})(L(d(x,\pa B_{R})))=\phi (L(d(x,\pa B_{R})))~\text{ for}~x\in\A_{r,R},
\end{equation}
where $\phi:= w\circ L^{-1}$.
Let us observe that since $u$ is strictly decreasing (resp., increasing) if $\ho>0$ (resp., $\ho<0$), $L$ is strictly increasing, and $d(\cdot,\pa B_{R})$ is strictly decreasing, we conclude that $\phi$ has a strict monotonicity opposite to that of $u$, according to the sign of $\ho$. 

Now we construct a test function $v$ for the variational characterization \eqref{eq:lambda1} of $\lambda_1^{\mathcal{NR}}(\Omega)$ as follows:
\begin{equation}
    v(x)=\phi (l(d(x,\pa \Omo)))
    \quad \text{for}~x\in \Om,
\end{equation}
where the function $l$ is defined in \eqref{tT_delta}.
Then, in the same vein as in the proof of Theorem~\ref{thm:NR1} and using Lemma~\ref{lem:paramet1}, we can derive the following relations:
\begin{gather}
    \int_{\Om}|\nabla v(x)|^2\dx =\int_{0}^{|\Om|}|\phi'(\alpha)|^2 h(\al)^2\da  \leq \int_{0}^{|\A_{r,R}|}|\phi'(\alpha)|^2 H(\al)^2 \da=\int_{\A_{r,R}}|\nabla u(x)|^2\dx,\\
\int_{\Om}v(x)^2\dx =\int_{0}^{|\Om|} \phi(\alpha)^2 \da=\int_{\A_{r,R}} u(x)^2\dx,\\
\int_{\pa \Omo} v(x)^2\dsi  =\left(\min_{\overline{\Om}}v\right)^2  |\pa \Omo|_1 =\left(\min_{\overline{\A_{r,R}}}u\right)^2|\pa B_R|_1=\int_{\pa B_R}u(x)^2\dsi, ~\text{ if}~\ho>0,\\
\int_{\pa \Omo} v(x)^2\dsi =\left(\max_{\overline{\Om}}v\right)^2 |\pa \Omo|_1 =\left(\max_{\overline{\A_{r,R}}}u\right)^2|\pa B_R|_1=\int_{\pa B_R}u(x)^2\dsi,~\text{ if}~\ho<0.
\end{gather}
We obtain the desired conclusion by applying these relations in the variational characterization \eqref{eq:lambda1} of $\la_1^{\mathcal{NR}}(\Om)$.
\qed

\subsection{Proof of Theorem~\ref{thm:RR} in the case \texorpdfstring{$\hi \cdot \ho = 0$}{h\_out*h\_in=0}} 
This case is covered by Theorem~\ref {thm:NR1} (for $\ho = 0$) and Theorem~\ref{thm:RN1} (for $\hi = 0$).
\qed

\section{Robin-Robin case}\label{sec:proof-RR}

The aim of this section is to prove Theorem~\ref{thm:RR} in the case $\hi \cdot \ho > 0$. 
In the following three subsections, we provide a required auxiliary material, and conclude the proof of Theorem~\ref{thm:RR} in the last subsection. 

In what follows, we always assume that $\hi,\ho \in \mathbb{R} \cup \{+\infty\}$ and $\hi \cdot \ho > 0$, and that $u$ is a positive first eigenfunction of \eqref{cutproblem2d} in $\Omega$; the normalization of $u$ is irrelevant for our analysis. 
Recall that $\lambda_1^{\mathcal{RR}}(\Omega) \neq 0$ by Remark~\ref{rem:sign}. 

\subsection{Critical points}\label{section:critpoints}
    Recall from Section~\ref{sec:intro} that $\Omega$ is of class $C^{1,1}$,  $u$ is a (real) analytic function in $\Omega$, and $u \in C^{1}(\overline{\Omega})$.
    We have the following two simple results. 
    \begin{lemma}\label{lem:critonboundary}
        Let $z_0 \in \partial\Om$. 
        If $\hi,\ho > 0$, then $\partial u(z_0)/\partial \nu < 0$, while if $\hi,\ho<0$, then $\partial u(z_0)/\partial \nu > 0$. 
        In particular, $u$ does not have  critical points on $\pa\Om$.
    \end{lemma}
    \begin{proof}
        In the case of the Dirichlet boundary condition, the result is a direct consequence of Hopf's lemma. Consider the Robin case.  
       If there is a point $z_0 \in \pa\Om$ such that $\partial u(z_0)/\partial \nu = 0$, then $u(z_0)=0$ due of the Robin boundary condition. Since $u$ is positive in $\Omega$, we derive a contradiction to Hopf's lemma, and hence $\partial u/\partial \nu \neq 0$ on $\pa\Omega$. 
       Then, the sign of $\partial u/\partial \nu$ alternates with the sign of $\hi, \ho$ by the Robin boundary condition.
    \end{proof}
    
    \begin{lemma}\label{cor:crit}
        If $\hi,\ho > 0$, then $u$ has a point of global maximum in $\Omega$
        and has no points of local minimum in $\Omega$, while if $\hi,\ho < 0$, then $u$ has a point of global minimum in $\Omega$ and has no points of local maximum in $\Omega$.
    \end{lemma}
    \begin{proof}
        Let $\hi,\ho > 0$. The existence of a global maximum in $\Omega$ is a consequence of Lemma~\ref{lem:critonboundary}. 
        If we suppose that $u$ has a point of local minimum $z_0 \in \Omega$, then $\Delta u(z_0) \geq 0$, which contradicts \eqref{cutproblem2d} since $u(z_0)>0$ and $\lambda_1^{\mathcal{RR}}(\Omega)>0$ by Remark~\ref{rem:sign}. 
        The case $\hi,\ho < 0$ is proved similarly.         
    \end{proof}

    Let $z_0 = (x_0,y_0) \in \Omega$ be a critical point of $u$. 
    We say that $u$ and a function $v: \mathbb{R}^2 \to \mathbb{R}$ at the points $z_0$ and $(0,0)$, respectively, are \textit{right-equivalent} if there exist neighborhoods $U_1$ of $z_0$ and $U_2$ of $(0,0)$ and a $C^1$-diffeomorphism $\phi: U_1 \to U_2$ such that $\phi(z_0) = (0,0)$ and $v(z) = u(\phi^{-1}(z))$ for any $z \in U_2$; cf. \cite[p.~58]{poston2014catastrophe}.

    Recalling that $u > 0$ in $\Omega$, we get the following classification of critical points of $u$, according to \cite[Section~2.1]{AVM}\footnote{The classification of critical points provided in \cite[Section~2.1]{AVM} is local, and hence, in view of Lemma~\ref{lem:critonboundary}, it is applicable to the first eigenfunction $u$ of \eqref{cutproblem2d}.}:    
\begin{enumerate}[\rm(I)]
\item $z_0$ is \textit{non-degenerate}, i.e., the Hessian determinant of $u$ at $z_0$ is not zero. 
In this case, $u$ is right-equivalent to one of the Morse functions
\begin{equation}\label{eq:nondegen}
	(x,y) \mapsto u(z_0) \pm x^2 \pm y^2.
\end{equation}

\item\label{classif:2} $z_0$ is \textit{semi-degenerate}, i.e., the Hessian matrix of $u$ at $z_0$ has rank $1$. 
In this case, $u$ is right-equivalent to one of the functions
\begin{align}
	\label{eq:semidegen1}
	&(x,y) \mapsto u(z_0) \pm y^2,\\
	\label{eq:semidegen2}
	&(x,y) \mapsto u(z_0) \pm x^k \pm y^2
	\quad \text{for a natural number}~ k \geq 3.	
\end{align}
\end{enumerate}
If we suppose that a critical point $z_0$ of $u$ is fully degenerate, i.e., the Hessian matrix of $u$ at $z_0$ is a zero matrix, then $\Delta u(z_0) = 0$. 
Hence, in view of \eqref{cutproblem2d} and Remark~\ref{rem:sign}, we get $u(z_0) = 0$, which contradicts the positivity of $u$ in $\Omega$. 

Thanks to this classification and in view of Lemma~\ref{cor:crit}, we have the following result on the structure of the set of critical points of $u$, cf.\ \cite[Lemma~2.13]{AVM}, \cite[Theorem~3.1]{arango2010critical}, and \cite{weineffect}.
\begin{lemma}\label{lem:isol} 
The set of critical points of $u$ is nonempty and consists of at most finitely many isolated critical points and at most one analytic curve $\theta$, which is a simple closed curve (i.e., a Jordan curve) and $u$ is constant along $\theta$. 
\end{lemma}
\begin{proof}
    Due to Lemma~\ref{cor:crit}, the set of critical points is nonempty. 
    The classification above implies that the number of isolated critical points is finite. Indeed, if there is a nonisolated critical point, then it necessarily belongs to a curve of critical points which consists of either local minimum critical points or local maximum critical points. 
    Since $u$ has no critical points on $\pa\Om$ (see Lemma~\ref{lem:critonboundary}), any curve of critical points lies in $\Omega$, and the analyticity of $u$ implies that this curve is simple and closed. 
    Assume that $\hi,\hi>0$. 
    According to Lemma~\ref{cor:crit}, any curve $\theta$ of critical points consists of local maximum critical points, and, at the same time, $u$ has no local minimum critical points. Hence, we deduce that $\theta$ must surround $\Omi$. 
    Since $\Omi$ is connected, we derive that there is at most one such $\theta$. 
    In view of the regularity of $u$, the Morse-Sard theorem implies that $u$ is constant along $\theta$. 
    The case $\hi,\hi<0$ is covered analogously. 
\end{proof}

\subsection{Gradient flow}
Since $u \in C^{1}(\overline{\Omega})$, it can be extended to a neighborhood of $\pa\Om$ in a $C^{1}$-way, see \cite[Lemma~6.37]{gilbargtrudinger}. 
The gradient \textit{flow lines} of $u$ are defined as solutions of the Cauchy problem
\begin{equation}\label{eq:gradflow2}
	\dot{z}(t) = -\nabla u(z(t)), \quad t \in I,
	\qquad z(0)=z_0 \in \overline{\Omega},
\end{equation}
where $z(t)=(x(t),y(t))$ and $I = I(z_0) \subset \mathbb{R}$ is the maximal interval such that the solution $\gamma(\cdot,z_0): I \to \mathbb{R}^2$ of \eqref{eq:gradflow2} belongs to $\overline{\Omega}$. 
Due to the regularity of $u$, $\gamma(\cdot,z_0)$ exists; moreover, $\gamma(\cdot,z_0)$ is unique and analytic in $\Omega$ whenever $z_0 \in \Omega$, see, e.g., \cite[Section~2.3, Remark~1]{perko2013differential}.
If $|\nabla u(z_0)| > 0$, then $u$ is strictly decreasing along $\gamma(\cdot,z_0)$, i.e., $u(\gamma(t_1,z_0)) > u(\gamma(t_2,z_0))$ for any admissible $t_1<t_2$, and $\partial u/\partial \nu = 0$ along $\gamma(\cdot,z_0)$. 
Let us explicitly note that we consider $\gamma(\cdot,z_0)$ for both positive and negative ``times'', with $\gamma(0,z_0) = z_0$. 
Trivially, if $z_0$ is a critical point of $u$, then $\gamma(t,z_0)=z_0$ for any $t \in \mathbb{R}$. 

The following simple lemma states that flow lines starting on the boundary $\partial\Omega$ remain in $\Omega$ for all positive or negative times, depending on the sign of the Robin parameters. 
\begin{lemma}\label{lem:traj:boundry}
    Let $z_0 \in \partial\Om$. 
    If $\hi,\ho > 0$, then $\gamma(t,z_0) \in \Omega$ for any $t \in [-\infty,0)$, while if $\hi,\ho<0$, then $\gamma(t,z_0) \in \Omega$ for any $t \in (0,+\infty]$. 
\end{lemma}
\begin{proof}
    Let $\hi,\ho > 0$. 
    By Lemma~\ref{lem:critonboundary}, for any $z_0 \in \partial\Om$ and any $\gamma(\cdot,z_0)$ we have 
    $$
    \dot \gamma(0,z_0) \cdot \nu
    =
    -\nabla u(\gamma(0,z_0)) \cdot \nu
    =
    -\partial u(z_0)/\partial \nu > 0.
    $$
    That is, $-\nabla u(z_0)$ makes an angle less than $\pi/2$ with $\nu$. 
    Since $\pa\Om$ is regular, we see that $\gamma(t,z_0) \in \Omega$ for any $t<0$ sufficiently close to $0$. 
    Recalling that $t \mapsto u(\gamma(t,z_0))$ is strictly monotone, it is not hard to conclude that $\gamma(t,z_0) \not\in \pa\Om$ for any $t < 0$. 
    Since $u$ has no critical points on $\pa\Om$ (see Lemma~\ref{lem:critonboundary}), we get $\gamma(t,z_0) \in \Omega$ for any $t \in [-\infty,0)$. 
    The case $\hi,\ho<0$ can be proved in the same way.  
\end{proof}   

\begin{remark}[Arnold-Thom conjecture]\label{rem:thom-conjecture}
In the case when $\gamma(\cdot,z_0)$ converges to a critical point of $u$ (either as $t \to +\infty$ or $t \to -\infty$), it is known from \cite[p.~768]{kurdyka2000proof} that the following limit exists:
$$
\lim_{t \to \infty} \frac{\dot{\gamma}(t,z_0)}{|\dot{\gamma}(t,z_0)|},
$$
that is, the normalized tangents to $\gamma(\cdot,z_0)$ converge. 
We note that the same result remains unknown in higher dimensions.
\end{remark}

The following technical fact, which states that secants of a flow line tend to be collinear to its normalized tangents, is geometrically evident from Remark~\ref{rem:thom-conjecture}, and we provide details for the sake of clarity. 
\begin{lemma}\label{lem:tangents}
	Let $z_0 \in \Omega$ be a regular point of $u$. 
	If ${\gamma}(t,z_0) = (x(t),y(t))$ converges to a critical point $z_1 = (x_1,y_1)\in \Omega$ as $t \to +\infty$ (resp.,~$t \to -\infty$), then the scalar product
	\begin{equation}\label{lem:tangents:claim}
		\left< \frac{(\dot x(t),\dot y(t))}{|\dot{\gamma}(t,z_0)|}, 
		\frac{(-(y(t)-y_1),x(t)-x_1)}{|{\gamma}(t,z_0)-z_1|}
		\right>
		\to 0 
	\end{equation}
	as $t \to +\infty$ (resp.,~$t \to -\infty$).
\end{lemma}
\begin{proof}
	Assume, without loss of generality, that ${\gamma}(t,z_0) \to z_1$ as $t \to +\infty$, where $z_1 = (x_1,y_1)\in \Omega$ is a critical point of $u$. 
	In view of Remark~\ref{rem:thom-conjecture}, there exists a vector $(a,b)$ with $|(a,b)|=1$ such that
	\begin{equation}\label{lem:tangents:1}
		\frac{(\dot x(t),\dot y(t))}{|\dot{\gamma}(t,z_0)|}  \to (a,b) 
		\quad \text{as}~ t \to +\infty.
	\end{equation}
	Rotating the coordinate system if necessary, we can assume that $b \neq 0$. 
	In particular, 
	\begin{equation}\label{lem:tangents:2}
		\frac{\dot x(t)}{\dot y(t)} \to \frac{a}{b} 
		\quad \text{as}~ t \to +\infty.
	\end{equation}
	On the other hand, by the extended mean value theorem, for any $t>0$ there exists $t_1 \in (t,+\infty)$ such that
	\begin{equation}\label{eq:lem46:1}
		(x(t) - x_1) \dot y(t_1)
		=
		(y(t) - y_1) \dot x(t_1).
	\end{equation}
	Hence, taking any sufficiently large $t>0$ and dividing both sides of \eqref{eq:lem46:1} by $|{\gamma}(t,z_0)-z_1| \, \dot y(t_1)$, we get
	$$
	\frac{x(t) - x_1}{|{\gamma}(t,z_0)-z_1|} 
	=
	\frac{y(t) - y_1}{|{\gamma}(t,z_0)-z_1|}  
	\cdot
	\frac{\dot x(t_1)}{\dot y(t_1)}.
	$$
	This implies that
	\begin{multline}
		\left<\frac{(\dot x(t),\dot y(t))}{|\dot{\gamma}(t,z_0)|},
		\frac{(-(y(t)-y_1),x(t)-x_1)}{|{\gamma}(t,z_0)-z_1|}
		\right>
		\\
		=
		\frac{y(t) - y_1}{|{\gamma}(t,z_0)-z_1|} 
		\left(
		-\frac{\dot x(t)}{|\dot{\gamma}(t,z_0)|}
		+ 
		\frac{\dot x(t_1)}{\dot y(t_1)} \cdot \frac{\dot y(t)}{|\dot{\gamma}(t,z_0)|} 
		\right).
		\label{eq:lem46:2}
	\end{multline}
	Taking into account \eqref{lem:tangents:1}, \eqref{lem:tangents:2}, and recalling that $t_1 \to +\infty$ as $t \to +\infty$, we deduce that the sum in the brackets on the right-hand side of \eqref{eq:lem46:2} tends to zero as $t \to +\infty$. 
	Thus, we obtain the convergence \eqref{lem:tangents:claim}.
\end{proof}

\subsection{Effectless cut}\label{sec:effectlesscup}
Let us consider the set of points in $\Omega$ whose flow lines reach $\pa \Omi$:
    \begin{align}\label{def:G}
        \Gi &=\{z\in \Om:~ \gamma(t_z, z) \in \pa\Omi~ \text{ for some}~
t_z \in \mathbb{R}\}.
    \end{align}
The part of the boundary $\pa \Gi \cap \Omega$ was studied by \textsc{Weinberger} \cite{weineffect} and it was called \textit{effectless cut} therein. The properties of the effectless cut are vital for the arguments of \textsc{Hersch} \cite[Section~3]{hers} and also for the proof of Theorem~\ref{thm:RR}. 
However, the analysis performed in \cite{weineffect} is very compressed and lacks various details. 
We revise the arguments from \cite{weineffect} and extend them to the present settings.

First, we show that $\Gi$ is open. 
We refer to \cite{weineffect} and \cite[Lemma~4.1]{AVM} for related statements. 

\begin{lemma}\label{lem:G:open}
	The set $\Gi$ is open, connected, and $\overline{\Gi} \cap \pa\Omo = \emptyset$.  
\end{lemma}
\begin{proof} 
	Let $\hi,\ho>0$. 
	We start with the openness of $\Gi$. 
	Take any $\xi \in \Gi$ and denote $z_0 := \gamma(t_\xi,\xi) \in \pa\Omi$, where $t_\xi > 0$ exists by the definition of $\Gi$. 
	By Lemma \ref{lem:critonboundary}, $|\nabla u(z_0)|>0$.
	
	We claim that for any $\varepsilon>0$ there exists $\delta>0$ such that if $z\in B_{\delta}(z_0)\cap \partial \Omega$, then any flow line $\gamma(\cdot,z)$ crosses the level set $\{u(z_0)+\varepsilon\}$ as $t$ decreases from $0$.
	(Here we note that the uniqueness of $\gamma(\cdot,z)$ cannot be guaranteed because of the lack of the boundary regularity of $u$.)
	
	Suppose for a moment that the claim is true. 
	Then we can find $z_1, z_2 \in \pa\Omi$ from either sides of $z_0$ such that  
	any two fixed flow lines $\gamma_i:=\gamma(\cdot,z_i)$, $i=1,2$, cross the level set $\{u(z_0)+\varepsilon\}$ as $t$ decreases from $0$. 
	Let now $U$ be a subdomain of $\Omega$ enclosed by $\gamma_1$, $\gamma_2$, $\partial \Omi$, and the level set $\{u(z_0)+\varepsilon_1\}$.
	Since $u$ is strictly monotone along the gradient flow, any flow line that reaches $U$ from $\Omega \setminus \overline{U}$ becomes trapped in $U$ and hence terminates on $\partial \Omi \cap \overline{U}$ in finite time.
	By the continuous dependence of the Cauchy problem \eqref{eq:gradflow2} on the initial data, there exist $\tilde\delta, \tilde t>0$ such that $\gamma(\tilde t,z)\in U$ for all $z\in B_{\tilde \delta}(\xi)$. 
	Consequently, $B_{\tilde \delta}(\xi) \subset \Gi$ and hence $\Gi$ is open.  
	Analogous arguments work in the case $\hi,\ho<0$. 
	
	Let us now verify the technical claim stated above. 
	First we show that for any  $\varepsilon>0$ there exist $\delta_0, t_0>0$ such that if $z\in B_{\delta_0}(z_0)\cap \partial\Omega$ and $t \in [-t_0,0]$, then for any flow line $\gamma(\cdot,z)$ we have 
	\begin{equation}\label{eq:flowincr:0}
		|\nabla u(\gamma(t,z))|^2 
		> 
		|\nabla u(z_0)|^2 - \varepsilon.
	\end{equation}
	Suppose, by contradiction, that there exists $\varepsilon >0$, a sequence $\{z_n\} \subset \partial \Omega$ converging to $z_0$, a sequence  of flow lines $\{\gamma_n(\cdot,z_n)\}$, and a sequence  $\{t_n\} \subset (-\infty,0)$ converging to $0$ such that 
	\begin{equation}\label{eq:flowincr1}
		|\nabla u(\gamma_n(t_n,z_n))|^2 
		\leq
		|\nabla u(z_0)|^2 - \varepsilon.
	\end{equation} 
    Notice that for any $z \in \overline{\Omega}$, $\gamma(\cdot,z)$, and $t<0$ we have the uniform estimate 
	$$
	|\gamma(t,z)-\gamma(0,z)| = \left|\int_t^0 \dot{\gamma}(s,z)\ds \right|=  \left|\int_t^0 -\nabla u(\gamma(s,z))\ds \right|\leq |t| \sup_{\eta\in \overline{\Omega}}\|\nabla u(\eta)\|.
	$$
	Therefore, the sequence $\{\gamma_n(t_n,z_n)\}$ converges to $z_0$. 
	The fact that $u \in C^1(\overline{\Omega})$ together with \eqref{eq:flowincr1} then yields a contradiction:
	$$
	|\nabla u(z_0)|^2 
	\leq
	|\nabla u(z_0)|^2 - \varepsilon.
	$$
	That is, \eqref{eq:flowincr:0} is satisfied. 
	Further, by the continuity of $u$ in $\overline{\Omega}$, 
	for any $\varepsilon_1>0$ we can find  $\delta_1>0$ such 
	that $\delta_1<\delta_0$ and 
	\begin{equation}\label{eq:flowincr:-1}
		u(z_0) + \varepsilon_1 > u(z) > u(z_0) - \varepsilon_1
		\quad \text{for any}~ z\in B_{\delta_1}(z_0)\cap \overline{\Omega}.
	\end{equation}
	At the same time, for any $t < 0$ we get 
	\begin{equation}\label{eq:flowincr:1}
		u(\gamma(t,z))
		=
		u(z) -\int_t^0 \frac{du}{ds}(\gamma(s,z)) \ds
		=
		u(z) + \int_t^0 |\nabla u(\gamma(s,z))|^2 \ds.
	\end{equation}
	Therefore, taking any $z\in B_{\delta_1}(z_0)\cap \partial \Omega$, $t \in [-t_0,0]$ and using \eqref{eq:flowincr:-1}, \eqref{eq:flowincr:0}, we deduce the estimate
	\begin{equation}\label{eq:flowincr:2}
		u(\gamma(t,z))
		>
		u(z_0) - \varepsilon_1
		-t(|\nabla u(z_0)|^2 - \varepsilon).
	\end{equation}
    Assuming that $\varepsilon>0$ is sufficiently small and choosing $\varepsilon_1 = t_0(|\nabla u(z_0)|^2 - \varepsilon) - \varepsilon > 0$, 
	for any $z\in B_{\delta_1}(z_0)\cap \partial \Omega$ we obtain 
	\begin{equation}\label{eq:flowincr:3}
		u(\gamma(-t_0,z))> u(z_0)+ \varepsilon,
	\end{equation}
	which gives the desired claim.

Let us prove the remaining assertions. 
It is not hard to deduce from Lemma~\ref{lem:traj:boundry} that $\Gi$ contains a neighborhood of $\partial \Omi$. 
Therefore, since $\pa\Omi$ is connected, so is $\Gi$. 
The claim $\overline{\Gi} \cap \pa\Omo = \emptyset$ also follows from Lemma~\ref{lem:traj:boundry} and the strict monotonicity of $u$ along the gradient flow. 
\end{proof}

Alongside $\Gi$, we consider the set of points whose flow lines reach $\pa \Omo$:
    \begin{align}\label{def:Go}
        \Go &=\{z\in \Om:~ \gamma(t_z, z) \in \pa\Omo~ \text{ for some}~
t_z \in \mathbb{R}\}.
    \end{align}
Arguing in much the same way as in Lemma~\ref{lem:G:open}, we have the following facts.
\begin{lemma}\label{lem:Gout:open}
    The set $\Go$ is open, connected, and $\overline{\Go} \cap \pa\Omi = \emptyset$.  
\end{lemma}

\begin{corollary}\label{cor:GG:empty-intersection}
    $\overline{\Gi} \cap \Go = \emptyset$ and $\Gi \cap \overline{\Go} = \emptyset$. 
\end{corollary}
\begin{proof}
   Lemmas~\ref{lem:G:open}, \ref{lem:Gout:open} and the strict monotonicity of $u$ along the gradient flow imply $\Gi \cap \Go = \emptyset$. 
   If, say, $\overline{\Gi} \cap \Go \neq \emptyset$, then there exists $z_0 \in \pa \Gi \cap \Go$. By the very definition of the boundary and by the openness of $\Go$, we can find $z_1 \in \Gi \cap \Go$, which is impossible.
\end{proof}

Using Lemma~\ref{lem:G:open} and arguing as in \cite[Lemma~4.4]{AVM} (see also \cite{weineffect}), we obtain the following result. 
We provide details for clarity. 
\begin{lemma}\label{lem:G:flow}
Let $z_0 \in \partial \Gi \cap \Omega$.
Then $\gamma(t,z_0) \in \partial \Gi \cap \Omega$ for any $t \in \mathbb{R} \cup \{\pm\infty\}$.
\end{lemma}
\begin{proof}
Since $\Gi$ is open, we have $z_0 \not\in \Gi$. 
Thus, $\gamma(\cdot,z_0)$ cannot reach $\pa\Omi$ in finite time. 
Thanks to Corollary~\ref{cor:GG:empty-intersection}, $\gamma(\cdot,z_0)$ cannot reach $\pa\Omo$ in finite time, too. 
Moreover, in view of Lemma~\ref{lem:critonboundary}, $\gamma(\cdot,z_0)$ cannot reach $\pa\Omi$ and $\pa\Omo$ as $t \to \pm\infty$ as well. 
That is, $\gamma(t,z_0) \in \Omega$ for any $t \in [-\infty,+\infty]$. 
Assume that $z_0$ is a regular point of $u$. 
For any $\delta>0$ and $t\in (-\infty,+\infty)$ consider $B_\delta(\gamma(t,z_0))$. By the definition of $\Gi$ and the continuous dependence of the Cauchy problem on the initial data,  we can observe that some of the flow lines emanating from $B_\delta(\gamma(t,z_0))$ terminate on $\partial\Omi$ and others do not terminate on $\partial\Omi$. Thus, $B_\delta(\gamma(t,z_0))$  intersects both $\Gi$ and $\Omega\setminus \Gi$, and hence $\gamma(t,z_0)\in \partial \Gi \cap \Om$ for any $t\in [-\infty,+\infty]$. 
\end{proof}

Our next aim is the following description of $\pa \Gi \cap \Omega$. 
We refer to \cite[Proposition~4.19]{AVM} for a related result, whose proof, however, is completely different. 
\begin{lemma}\label{lem:G:boundary}
$\partial \Gi \cap \Omega$ consists of finitely many flow lines (together with their end-points), arcs of critical points, and isolated critical points. 
\end{lemma}
\begin{proof}
Due to Lemma~\ref{lem:G:flow}, the flow line emanating from any regular point on $\pa \Gi \cap \Omega$ fully belongs to $\pa \Gi \cap \Omega$.
Suppose that the number of such flow lines is infinite. 
Denote them as $A_1B_1$, $A_2B_2$, $\dots$, where the end-points $A_i$, $B_i$ are critical points of $u$. 
According to Lemma~\ref{lem:isol}, $u$ has at most finitely many isolated critical points and at most one Jordan curve $\theta$ of critical points, and $u$ is constant along $\theta$. 
Therefore, up to relabeling, we have two cases:
\begin{enumerate}
    \item There are two isolated critical points $z_1$, $z_2$, and $N \geq 1$ such that $A_i = z_1$, $B_i = z_2$ for all $i \geq N$.
    \item There is an isolated critical point $z_1$ and $N \geq 1$ such that $A_i =z_1$, $B_i \in \theta$ for all $i \geq N$.
\end{enumerate}
Since $A_iB_i \subset \pa \Gi$ for each $i$, we get a contradiction to the openness and connectedness of $\Gi$ (see Lemma~\ref{lem:G:open}). 
\end{proof}

The same facts as in Lemmas~\ref{lem:G:flow} and \ref{lem:G:boundary} remain valid for $\pa \Go \cap \Omega$. 

In order to avoid possible isolated critical points and cracks in $\partial \Gi \cap \Om$ and $\partial \Go \cap \Om$, we regularize $\Gi$ and $\Go$ by considering the sets
\begin{equation}\label{eq:g-star}
\Gi^* := \text{Int}(\overline{\Gi})
\quad \text{and} \quad 
\Go^* := \text{Int}(\overline{\Go}).
\end{equation}
Clearly, we have $\pa \Gi^* \subset \pa \Gi$ and $\pa \Go^* \subset \pa \Go$. 
\begin{lemma}\label{lem:G:doubly-connected}
$\partial \Gi^* \cap \Omega = \partial \Go^* \cap \Omega$ and $\Gi^*$, $\Go^*$ are doubly connected domains.
\end{lemma}
\begin{proof}
If we suppose that $\partial \Gi^* \cap \Omega \neq \partial \Go^* \cap \Omega$, then it is not hard to see that $\Omega \setminus \overline{(\Gi^* \cup \Go^*)}$ is nonempty. 
Take any connected component $\Omega_0$ of this set and note that its boundary is a subset of $(\partial \Gi^* \cup \partial \Gi^*) \cap \Omega$. 
We deduce from Lemma~\ref{lem:G:boundary} that $\pa {\Omega}_0$ is a finite union of piecewise smooth Jordan curves and, in view of the properties of the gradient flow, $\partial u/\partial \nu = 0$ on each smooth piece of $\partial {\Omega}_0$. 
Integrating the equation in \eqref{cutproblem2d} over ${\Omega}_0$, recalling that $\overline{\Omega_0} \subset \Omega$ and $u \in C^\infty(\Omega)$, and applying the divergence theorem (see, e.g., \cite[Theorem~9.6]{maggi2012sets}), we obtain
$$
\lambda_1^\mathcal{RR}(\Omega)
\int_{{\Omega}_0} u \dx
=
-
\int_{{\Omega}_0} \Delta u \dx
=
-
\int_{\partial {\Omega}_0}
\frac{\partial u}{\partial \nu} \dS
=
0.
$$
Since $\lambda_1^\mathcal{RR}(\Omega) \neq 0$ (see Remark~\ref{rem:sign}), 
we get a contradiction to the positivity of $u$.  
Therefore, we have $\partial \Gi^* \cap \Omega = \partial \Go^* \cap \Omega$.

It follows from Lemma~\ref{lem:G:open} that $\Gi^*$ is a domain. Moreover, it is not hard to conclude from Lemma~\ref{lem:critonboundary} that $\Omi$ is a ``hole'' in $\Gi^*$. 
Similar facts hold for $\Go^*$. 
Thus, if we suppose that, say, $\Gi^*$ has another ``hole'', then we get a contradiction to the equality $\partial \Gi^* \cap \Omega = \partial \Go^* \cap \Omega$. 
Since $\pa \Gi^*$, $\pa \Go^*$ do not contain isolated points and cracks by construction, we conclude that $\Gi^*$, $\Go^*$ are doubly connected domains.
\end{proof}

\begin{corollary}
    $\Go^* = \Omega \setminus \overline{\Gi^*}$.
\end{corollary}

Finally, with a slight abuse of notation, we define the \textit{effectless cut} of $\Omega$ as the set\footnote{At the end of his arguments, \textsc{Weinberger} \cite{weineffect} also removes isolated critical points from $\pa \Gi \cap \Omega$. Here we propose to avoid possible cracks on $\pa \Gi$, as well.} 
\begin{equation}\label{eq:eff-star}
E = \partial \Gi^* \cap \Omega. 
\end{equation}
In view of Lemmas~\ref{lem:G:open}, \ref{lem:G:boundary}, and \ref{lem:G:doubly-connected}, we have the following result.
\begin{lemma}\label{lem:E:properties}
    The effectless cut $E$ satisfies the following properties: 
    \begin{enumerate}[(1)]
    	\item $E = \partial \Go^* \cap \Omega$.
        \item\label{lem:E:properties:1} $E$ is a simple closed curve (i.e., a Jordan curve) and it can be decomposed as
\begin{equation}\label{eq:E:decomp}
E = \bigcup_{i = 1}^n E_i,
\end{equation}
where $n \geq 1$ and each $E_i$ is a smooth arc of finite length, which is either a flow line (together with its end-points) or an arc of critical points of $u$. 
If $E_i$ intersects $E_j$, then the intersection consists of at most two points, each of which is a critical point of $u$.
      \item\label{lem:E:properties:2} $E \subset \Omega$ and $\Omi$ is compactly contained in a domain bounded by the Jordan curve $E$.
    \end{enumerate}
\end{lemma}

\subsection{Proof of Theorem~\ref{thm:RR} in the case \texorpdfstring{$\hi \cdot \ho > 0$}{h\_out*h\_in>0}}\label{sec:proof_RR-2}

Our first aim is to prove that
\begin{equation}\label{eq:proof1:1}
	\lambda_1^{\mathcal{RR}}(\Om) 
	\leq
	\lambda_1^{\mathcal{RN}}(\Gi^*)
	\quad \text{and} \quad 
	\lambda_1^{\mathcal{RR}}(\Om) 
	\leq
	\lambda_1^{\mathcal{NR}}(\Go^*),
\end{equation}
where
$$
\lambda_1^{\mathcal{NR}}(\Go^*) 
=
\inf_{v \in \widetilde{H}^1(\Go^*)\setminus\{0\}}
\frac{\int_{\Go^*}|\nabla v|^2\dx+\ho \int_{\pa \Omo} v^2\dsi}{\int_{\Go^*} v^2\dx},
$$
and $\lambda_1^{\mathcal{RN}}(\Gi^*)$ is defined analogously. 

Let us prove, for instance, the second inequality in \eqref{eq:proof1:1}. 
We start by discussing the regularity of the effectless cut $E = \pa \Go^* \cap \Om = \pa \Gi^* \cap \Om$. 
According to Lemma~\ref{lem:E:properties}, $E$ has at most a finite number of points $z_1, \dots, z_n$ (which are necessarily critical points of $u$) at which the Lipschitzness of $\Go^*$ can be lost. We investigate it in more details. 

Let $E_1$ and $E_2$ be any two analytic arcs from the decomposition \eqref{eq:E:decomp} of $E$ which intersect each other at a critical point $z \in \Omega$.
If one of them is a flow line and the other is an arc of critical points of $u$, then, thanks to \cite[Lemma~2.27 (i)]{AVM}, $E_1$ meets $E_2$ at $z$ orthogonally. 
Thus, in this case, $\Go^*$ is Lipschitz in a neighborhood of $z$. 
If both $E_1$ and $E_2$ are arcs of critical points, then they belong to the same curve of critical points (see Lemma~\ref{lem:isol}), and hence $\Go^*$ is smooth in a neighborhood of $z$.
Assume now that $E_1$ and $E_2$ are two flow lines. 
In view of Remark~\ref{rem:thom-conjecture}, we can define the angle $\vartheta \in [0,2\pi]$ between the limits of normalized tangents to $E_1$ and $E_2$ at $z$. 
If $\vartheta \in (0,2\pi)$, then $\Go^*$ is Lipschitz in a neighborhood of $z$. 
The cases $\vartheta = 0$ and $\vartheta = 2\pi$ correspond to the outer cusp and the inner cusp, respectively. 
Lemma~\ref{lem:tangents} shows that for any sufficiently small $\varepsilon>0$ the curve $\partial B_\varepsilon(z)$ intersects $E_1$ and $E_2$ at positive angles approaching $\pi/2$ as $\varepsilon \to 0$. 
Notice, moreover, that $\partial B_\varepsilon(z)$ intersects $E_i$ only at a single point for each sufficiently small $\varepsilon>0$. Indeed, if we suppose that there is more than one point of intersection, then, by the continuity arguments, there will be $\varepsilon_1 \in (0,\varepsilon)$ at which $\partial B_{\varepsilon_1}(z)$ intersects $E_i$ tangentially, which contradicts Lemma~\ref{lem:tangents}.

Let $z_1, \dots, z_k$ be the intersection points of the flow lines from the decomposition \eqref{eq:E:decomp} of $E$ at which the Lipschitzness of $\Go^*$ fails. 
Then we conclude from the above discussion that the set 
$$
\Go^*(\varepsilon):= \Go^* \setminus \bigcup_{i = 1}^k B_\varepsilon(z_i)
$$
is a Lipschitz domain for any sufficiently small $\varepsilon>0$.
As a consequence, according to \cite[Theorem~8]{smith1994smooth}, 
$C^\infty(\mathbb{R}^2)$ is dense in $H^1(\Go^*)$, which means that $C^\infty(\overline{\Go^*})$ is dense in $H^1(\Go^*)$. 
Moreover, arguing via the partition of unity, it is not hard to see that, in the case $\ho = +\infty$, $C^\infty(\overline{\Go^*})$-functions vanishing in a neighborhood of $\pa\Omo$ are dense in $\widetilde{H}^1(\Go^*)$. 

Take any $v \in C^\infty(\overline{\Go^*}) \setminus \{0\}$. 
In the case $\ho = +\infty$, we additionally assume that $v$ vanishes in a neighborhood of $\pa\Omo$. 
Since $u \in C^\infty(\Omega)$ and $u>0$ in $\Omega$, the classical Picone identity gives the pointwise estimate  in $\Go^*$:
\begin{equation}\label{eq:pic0}
|\nabla v|^2 - \left< \nabla u, \nabla \left(\frac{v^2}{u}\right)\right>
=
|\nabla v|^2 + \frac{v^2}{u^2} |\nabla u|^2 - 2 \frac{v}{u} \left<\nabla u, \nabla v \right>
=
\left|
\nabla v - \frac{v}{u} \nabla u
\right|^2
\geq 0. 
\end{equation}
Integrating over $\Go^*$, we get
\begin{equation}\label{eq:ibp2x0}
\int_{\Go^*}
\left<\nabla u, \nabla \left(\frac{v^2}{u}\right)\right> \dx
\leq
\int_{\Go^*}
|\nabla v|^2 \dx.
\end{equation}
Now we would like to integrate by parts the left-hand side of \eqref{eq:ibp2x0}. 
However, the integrand might have insufficient regularity to directly apply \cite[Theorem~9.6]{maggi2012sets}. 
We use approximation arguments instead. 
Namely, it is not hard to see that $u, v^2/u \in H^1(\Go^*)$: in the case $\ho<+\infty$, we have $u>0$ in $\overline{\Go^*}$ by Hopf's lemma, and in the case $\ho=+\infty$, we recall that $v$ vanishes in a neighborhood of $\pa\Omo$, implying $v^2/u \in C^1(\overline{\Go^*})$. 
Consequently, we are able to integrate by parts over the Lipschitz domain $\Go^*(\varepsilon)$ for any sufficiently small $\varepsilon>0$ (see, e.g., \cite[Theorem~4.4]{McLean}) and get
\begin{equation}\label{eq:ibp}
	\int_{\Go^*(\varepsilon)}
	\left<\nabla u, \nabla \left(\frac{v^2}{u}\right)\right> \dx
	=
	-\int_{\Go^*(\varepsilon)} \Delta u  \left(\frac{v^2}{u}\right) \dx
	+
	\int_{\pa \Go^*(\varepsilon)} \frac{\partial u}{\partial \nu} \left(\frac{v^2}{u}\right) \dsi.
\end{equation}
In view of the problem \eqref{cutproblem2d}, we arrive at
\begin{align}
	\int_{\Go^*(\varepsilon)}
	\left<\nabla u, \nabla \left(\frac{v^2}{u}\right)\right>  \dx
	=
	\lambda_1^{\mathcal{RR}}(\Omega) \int_{\Go^*(\varepsilon)} v^2 \dx
	& -
	 \ho\int_{{\pa \Omo}} v^2 \dsi\\
	& +
	\int_{\pa \Go^*(\varepsilon) \cap \Omega} \frac{\partial u}{\partial \nu} \left(\frac{v^2}{u}\right) \dsi.\label{eq:ibp2}
\end{align}
Observe that the absolute continuity of the Lebesgue integral implies
$$
\int_{\Go^*(\varepsilon)} v^2 \dx 
\to
\int_{\Go^*}
v^2 \dx
\quad \text{as } \varepsilon \to 0.
$$
Moreover, since $v$ is bounded, $u \in C^1(\overline{\Omega})$, and $\partial u/\partial\nu = 0$ on every regular piece of $E$, we get
\begin{align}
&\left|
\int_{\pa \Go^*(\varepsilon) \cap \Omega} \frac{\partial u}{\partial \nu} \left(\frac{v^2}{u}\right) \dsi
\right|
=
\left|
\int_{\Go^* \cap \bigcup_{i = 1}^k \partial B_\varepsilon(z_i)} \frac{\partial u}{\partial \nu} \left(\frac{v^2}{u}\right) \dsi
\right|
\\
&\leq
\frac{\max_{\Omega} (|\nabla u| v^2)}{\min_{\bigcup_{i = 1}^k B_\varepsilon(z_i)} u}
\cdot 
\left|\Go^* \cap \bigcup_{i = 1}^k \partial B_\varepsilon(z_i)\right|_1 
\leq 
C \cdot \bigcup_{i = 1}^k \left|\partial B_\varepsilon(z_i)\right|_1 \to 0
\quad \text{as } \varepsilon \to 0.
\end{align}
On the other hand, similarity to \eqref{eq:ibp2x0}, the Picone identity \eqref{eq:pic0} gives
\begin{equation}\label{eq:ibp2x}
\int_{\Go^*(\varepsilon)}
\left<\nabla u, \nabla \left(\frac{v^2}{u}\right)\right> \dx
\leq
\int_{\Go^*(\varepsilon)}
|\nabla v|^2 \dx
\leq 
\int_{\Go^*}
|\nabla v|^2 \dx.
\end{equation}
Therefore, we conclude from \eqref{eq:ibp2} and \eqref{eq:ibp2x} that
$$
\lambda_1^{\mathcal{RR}}(\Omega) \int_{\Go^*} v^2 \dx
-
\ho\int_{\pa \Omo} v^2 \dsi
+ C(\varepsilon)
\leq
\int_{\Go^*}
|\nabla v|^2 \dx,
$$
where $C(\varepsilon) \to 0$ as $\varepsilon \to 0$, which yields
$$
\lambda_1^{\mathcal{RR}}(\Omega) 
\leq
\frac{\int_{\Go^*}
|\nabla v|^2 \dx
+
\ho\int_{\pa \Omo} v^2 \dsi}{\int_{\Go^*} v^2 \dx}. 
$$
Since $v \in C^\infty(\overline{\Go^*}) \setminus \{0\}$ is arbitrary and $C^\infty(\overline{\Go^*})$ is dense in $\widetilde{H}^1(\Go^*)$, we get the desired inequality 
$$
\lambda_1^{\mathcal{RR}}(\Om) 
\leq
\lambda_1^{\mathcal{NR}}(\Go^*).
$$
The same proof works for $\Gi^*$, which establishes \eqref{eq:proof1:1}. 
In fact, taking $u$ as a test function for $\lambda_1^{\mathcal{NR}}(\Go^*)$, $\lambda_1^{\mathcal{NR}}(\Gi^*)$ and integrating by parts over approximating domains as above, we also obtain the converse inequalities to \eqref{eq:proof1:1}. 
Therefore, 
\begin{equation}\label{eq:proof1:1x}
	\lambda_1^{\mathcal{RR}}(\Om) 
	=
	\lambda_1^{\mathcal{RN}}(\Gi^*)
	\quad \text{and} \quad 
	\lambda_1^{\mathcal{RR}}(\Om) 
	=
	\lambda_1^{\mathcal{NR}}(\Go^*).
\end{equation}

By Theorems~\ref{thm:NR1} and \ref{thm:RN1}, respectively, there exist $\sigma_1, \sigma_2 \in (r,R)$ such that  
\begin{gather}
	\label{eq:proof1:5}
	\lambda_1^{\mathcal{RN}}(\Gi^*)
	\leq 
	\lambda_1^{\mathcal{RN}}(\A_{r,\sigma_1}),
	\quad \text{where}~
	|\Gi^*|=|\A_{r,\sigma_1}|,~ |\pa \Om_{\text{in}}|_1=|\pa B_r|_1,\\
	\label{eq:proof1:4}
	\lambda_1^{\mathcal{NR}}(\Go^*)
	\leq 
	\lambda_1^{\mathcal{NR}}(\A_{\sigma_2,R}),
	\quad \text{where}~ 
	|\Go^*|=|\A_{\sigma_2,R}|,~ |\pa \Om_{\text{out}}|_{1}=|\pa B_R|_{1}.
\end{gather}
Since $|\Gi^*|+|\Go^*| = |\Omega| = |\A_{r,R}|$, we conclude that $\sigma_1=\sigma_2$. 
On the other hand, Lemma~\ref{lem:character} gives
$$
\lambda_1^{\mathcal{RR}}(\A_{r,R}) 
=
\max_{\delta \in (r,R)}
\min \{
\lambda_1^{\mathcal{RN}}(\A_{r,\delta}),
\lambda_1^{\mathcal{NR}}(\A_{\delta,R})
\},
$$
which implies that
\begin{equation}\label{eq:proof1:6}
	\min \{
\lambda_1^{\mathcal{RN}}(\A_{r,\sigma_1}),
\lambda_1^{\mathcal{NR}}(\A_{\sigma_1,R})
\}
	\}
	\leq 
	\lambda_1^{\mathcal{RR}}(\A_{r,R}). 
\end{equation}
Combining \eqref{eq:proof1:1} (or \eqref{eq:proof1:1x}), \eqref{eq:proof1:4}, \eqref{eq:proof1:5}, \eqref{eq:proof1:6}, we arrive at
\begin{align}
	\lambda_1^{\mathcal{RR}}(\Om) 
	\leq 
	\lambda_1^{\mathcal{RR}}(\A_{r,R}),
\end{align}
which finishes the proof of Theorem~\ref{thm:RR} in the case $\hi \cdot \ho > 0$.  
\qed

\section{Final remarks}\label{sec:remarks}

\begin{enumerate}
    \item As indicated in Table~\ref{tab1}, results in the spirit of Theorem~\ref{thm:RR} for the Robin parameters of different signs, i.e., $\hi \cdot \ho <0$, remain largely open, except for a ``limiting'' case studied in \cite{krejvcivrik2024optimisation,krejvcivrik2020optimisation2}. 
    
    \item 
    We anticipate that equalities should hold in Theorems~\ref{thm:NR1} and \ref{thm:RN1} (and, consequently, in Theorem~\ref{thm:RR}) 
    if and only if $\Omega = \A_{r,R}$. 
    However, the present arguments do not allow to conclude it directly due to the lack of rigidity in Nagy's inequalities (see Remark~\ref{rem:Nagy}), in contrast to the higher-dimensional case (cf.\ \cite[Remark~3.3]{AnoopMrityunjoy}). 
    We refer to \cite{cito2024stability} for a related development. 
    
    \item The arguments of \cite{weineffect} allow additional ``holes'' in $\Omega$, and \cite{hers} allows additional ``holes'' with the Neumann boundary conditions. 
    We do not pursue this generalization herein for the clarity of exposition, but the present analysis can be adopted for that. 
    At the same time, we do not know whether Theorem~\ref{thm:RR} can be generalized to more ``holes'' in $\Omega$ with non-Neumann boundary conditions.
     
    \item Assume that $\Omega$ is, say, a Lipschitz domain, and let $\{\Omega_n\}$ be a family of $C^{1,1}$-domains approximating $\Omega$ in a way that ensures $\lambda_1^\mathcal{RR}(\Omega_n) \to \lambda_1^\mathcal{RR}(\Omega)$. 
    Then the result of Theorem~\ref{thm:RR} generalizes to such less regular $\Omega$. 
    These spectral stability results are known in the case of positive Robin parameters (see, e.g., \cite{burenkov2008spectral} and \cite[pp.~783-784]{daners2006faber}), but we were unable to find the corresponding literature in the case $\hi,\ho < 0$. 
\end{enumerate}

	\bigskip
	\noindent
	\textbf{Acknowledgments.}
        A significant part of this work was carried out during a series of visits of V.B.\ and M.G.\ to IIT Madras  supported by the Office of Global Engagement of the IIT Madras.
        T.V.A.~also acknowledges the Core Research Grant  (CRG/2023/005344)  by ANRF.  
        V.B.~was supported in the framework of the development program of the Scientific Educational Mathematical Center of the Volga Federal District (agreement No.\ 075-02-2025-1637). 
        M.G.~is supported by TIFR Centre for Applicable Mathematics (TIFR-CAM).

\bibliographystyle{abbrvurl}
\bibliography{Reference}

\end{document}